\documentclass[12pt,reqno]{amsart}
\usepackage[a4paper,margin=2.4cm]{geometry}
\usepackage{tikz}
\usepackage{amssymb,amsmath,amsthm,enumerate}

\usepackage[colorlinks=true, linkcolor=blue, urlcolor=magenta, hypertexnames=false]{hyperref}

\newcommand{\N}{{\mathbb N}}
\newcommand{\Z}{{\mathbb Z}}
\newcommand{\R}{{\mathbb R}}
\newcommand{\C}{{\mathbb C}}

\renewcommand{\H}{\mathcal{H}}

\newcommand{\dd}{\mathrm d}
\newcommand{\ii}{\mathrm i}
\newcommand{\ee}{\mathrm e}

\newcommand{\g}{{\mathfrak{g}}}

\numberwithin{equation}{section}

\theoremstyle{plain}
\newtheorem{theorem}{\bf Theorem}[section]
\newtheorem*{theorem*}{Theorem}
\newtheorem{lemma}[theorem]{\bf Lemma}
\newtheorem{proposition}[theorem]{\bf Proposition}
\newtheorem*{proposition*}{\bf Proposition}

\newtheorem{conjecture}[theorem]{Conjecture}

\theoremstyle{definition}
\newtheorem{definition}[theorem]{\bf Definition}
\newtheorem*{definition*}{\bf Definition}

\theoremstyle{remark}
\newtheorem*{remark*}{\bf Remark}
\newtheorem{remark}[theorem]{\bf Remark}

\newcommand{\ddiv}{\mathrm{div}}
\newcommand{\shift}{\mathrm{S}}
\newcommand{\midd}{\mathrm{M}}
\newcommand{\id}{\mathrm{I}}
\newcommand{\grad}{{\nabla}}
\newcommand{\lap}{{\Delta}}

\newcommand{\brac}[1]{\langle #1 \rangle}

\newcommand{\bigO}[1]{\mathcal{O}\left(#1\right)}
  
\begin{document}

\title[Optimal discrete $p$-Hardy--Rellich--Birman inequalities]{Optimal discrete $p$-Hardy--Rellich--Birman inequalities}

\author{Franti\v{s}ek \v{S}tampach}
\address{Department of Mathematics, Faculty of Nuclear Sciences and Physical Engineering, Czech Technical University in Prague, Trojanova 13, 12000 Prague~2, Czech Republic.}
\email{stampfra@cvut.cz}

\author{Jakub Waclawek}
\address{Department of Mathematics, Faculty of Nuclear Sciences and Physical Engineering, Czech Technical University in Prague, Trojanova 13, 12000 Prague~2, Czech Republic.}
\email{waclajak@cvut.cz}



\date{\today}

\begin{abstract}
We present a theory for constructing optimal lower bounds for the discrete half-line $p$-Laplacian of higher order $\ell\in\N$ and general $p>1$. The abstract framework introduces higher-order monotonicity and asymptotic constraints on a parameter sequence that determines optimal weights. As a concrete application, we specialize the parameter sequence to deduce new optimal discrete $p$-Hardy ($\ell=1$), $p$-Rellich ($\ell=2$), and $p$-Birman ($\ell\geq 3$) inequalities.
\end{abstract}

\maketitle

\section{Introduction} 

\subsection{From Hilbert and Hardy to optimal weights}

Hardy's journey toward his famous inequality began as early as 1915. In its simplest ($p=2$) difference form, Hardy's inequality states that
\begin{equation}
  \sum_{n=1}^\infty |u_n-u_{n-1}|^2
    \geq \sum_{n=1}^{\infty}w_{2}^{\mathrm{H}}(n)|u_n|^2, \quad\text{with}\;\; w_{2}^{\mathrm{H}}(n)=\frac{1}{4n^{2}},
      \label{eq:2-Hardy_dis}
\end{equation}
for all $u\in\ell^{2}(\N_0)$ with $u_0=0$.
Hardy originally sought a more elementary or elegant proof for another prominent inequality due to Hilbert. Hilbert's double-sum inequality~\cite{hil_89}, with its optimal constant $\pi$ verified by Schur~\cite{sch_11} in 1911, reads
\begin{equation}
\sum_{m,n=0}^{\infty}\frac{u_{m}\overline{u_{n}}}{m+n+1}\leq\pi\sum_{n=0}^{\infty}|u_n|^{2}
\label{eq:Hilbert_ineq}
\end{equation}
for all $u\in\ell^{2}(\N_0)$, where the bar denotes complex conjugation. We refer the reader to \cite{kuf-mal-per_06, kuf-mal-per_07} for a thorough account of the pre-history of Hardy and related inequalities.

It was already known to Schur~\cite{sch_11} that, despite the optimality of the constant $\pi$, the weight in Hilbert's inequality admits an improvement. Specifically, the value $1$ in the denominator of~\eqref{eq:Hilbert_ineq} can be taken as small as $1/2$. In fact, Schur's method yields the generalized Hilbert inequality
\[
\sum_{m,n=0}^{\infty}\frac{u_{m}\overline{u_{n}}}{m+n+\lambda}\leq\pi\sum_{n=0}^{\infty}|u_n|^{2}
\]
for all $\lambda\geq1/2$. If $\lambda$ is even smaller than $1/2$, the constant $\pi$ must be replaced by the larger sharp constant $\pi/|\sin\pi\lambda|$. These optimal constants are precisely the operator norms of the generalized Hilbert matrix on $\ell^{2}(\N_0)$, whose complete spectral representation is well known today; see \cite{ros_58} or \cite[Thm.~8]{kal-sto_16}.

Surprisingly, no similar improvements of Hardy's inequality~\eqref{eq:2-Hardy_dis} appeared in the literature for decades, despite the great interest in the topic. Only very recently, Keller, Pinchover, and Pogorzelski showed in~\cite{kel-pin-pog_cmp18} (see also \cite{kel-pin-pog_amm18, kre-lap-sta_22, kre-sta_22}) that the Hardy weight $w_{2}^{\mathrm{H}}$ in~\eqref{eq:2-Hardy_dis} can be replaced by the strictly larger weight
\begin{equation}
 w_{2}^{\mathrm{KPP}}(n)=2-\sqrt{1-\frac{1}{n}}-\sqrt{1+\frac{1}{n}}.
\label{eq:kpp-weight}
\end{equation}
They also proved that $w_{2}^{\mathrm{KPP}}$ is an optimal Hardy weight (see Definition~\ref{def:opt} below), which means, in particular, that it does not allow for further pointwise improvement.

As a consequence of our main result, we discover another improvement of the classical Hardy weight $w_{2}^{\mathrm{H}}$ that is more reminiscent of the generalized Hilbert inequality introducing a parameter into the denominator. Namely, it follows from Theorem~\ref{thm:optimal_rho_LB_form} (with $p=2$ and $\ell=1$) that the inequality 
\begin{equation}
  \sum_{n=1}^\infty |u_n-u_{n-1}|^2
    \geq \frac{1}{4}\sum_{n=1}^{\infty}\frac{|u_n|^2}{n(n+\lambda)}
      \label{eq:2-Hardy_dis_lam}
\end{equation}
holds for all $\lambda\geq-1/2$ and $u\in\ell^{2}(\N_0)$ with $u_0=0$. Furthermore, we prove that this weight with $\lambda=-1/2$ is also optimal.

The $p$-generalization of~\eqref{eq:2-Hardy_dis} with a sharp constant, i.e., the $p$-Hardy inequality
\begin{equation}
  \sum_{n=1}^\infty |u_n-u_{n-1}|^p
    \geq \sum_{n=1}^{\infty}w_{p}^{\mathrm{H}}(n)|u_n|^p,
    \quad\text{with}\;\; w_{p}^{\mathrm{H}}(n)=\left(\frac{p-1}{p}\right)^{\!p}\frac{1}{n^{p}},
      \label{eq:p-Hardy_dis}
\end{equation}
which holds for all $p>1$ and $u\in\ell^{p}(\N_0)$ with $u_0=0$, was proved by Landau in a 1921 letter to Hardy. Recently, Fischer, Keller, and Pogorzelski~\cite{fis-kel-pog_23} established a $p$-extension of~\eqref{eq:kpp-weight} by proving that the sequence
\begin{equation}\label{eq:weight_fkp}
w_{p}^{\mathrm{FKP}}(n)=\left(1-\left(1-\frac{1}{n}\right)^{\!\frac{p-1}{p}}\right)^{\!p-1}-\left(\left(1+\frac{1}{n}\right)^{\!\frac{p-1}{p}}-1\right)^{\!p-1},
\end{equation}
serves as an optimal weight into the $p$-Hardy inequality~\eqref{eq:p-Hardy_dis} for real-valued sequences $u$, and improves upon $w_{p}^{\mathrm{H}}$, see also~\cite{sta-wac_26}. To our knowledge, this is the only optimal discrete $p$-Hardy weight known until now. In Theorem~\ref{thm:optimal_rho_LB_form} (with $\ell=1$), we discover a new optimal $p$-Hardy weight
\begin{equation}
  w_{p}^{\mathrm{SW}}(n)=\left(\frac{p-1}{p}\right)^{\!p-1}\left(\frac{1}{(n-1/p)^{p-1}}-\frac{1}{n^{p-1}}\right),
      \label{eq:p-Hardy_dis_opt}
\end{equation}
which also improves pointwise upon the classical weight $w_{p}^{\mathrm{H}}$ and, for $p=2$, simplifies to the weight in~\eqref{eq:2-Hardy_dis_lam} with $\lambda=-1/2$.

\subsection{Beyond first order: the Rellich--Birman legacy}

The main focus of this article is on optimal higher-order variants of discrete $p$-Hardy inequalities, namely the optimal discrete $p$-Rellich and $p$-Birman inequalities. These inequalities first appeared in their continuous form for the case $p=2$.

A more renowned continuous version of the $p$-Hardy inequality appeared shortly after its discrete counterpart~\eqref{eq:p-Hardy_dis}, see~\cite{har_25}. Its differential form states that
\[
  \int_0^\infty |u'(x)|^p\,\dd x
    \geq \left(\frac{p-1}{p}\right)^{\!p} \int_0^\infty \frac{|u(x)|^p}{x^p}\dd x
\]
for all $p>1$ and $u\in W^{1,p}(\R_+)$ with $u(0_{+})=0$; here and below $\R_{+}\equiv(0,\infty)$.
An analogous lower bound for the $p$-norm of the second derivative,
\[
  \int_0^\infty |u''(x)|^p\,\dd x
    \geq \frac{(p-1)^{p}(2p-1)^{p}}{p^{2p}} \int_0^\infty \frac{|u(x)|^p}{x^{2p}}\dd x,
\]
holds for all $p>1$ and $u\in W^{2,p}(\R_+)$ with $u(0_{+})=u'(0_{+})=0$. 
This inequality is referred to as the $p$-Rellich inequality on $\R_{+}$, see~\cite{dav-hin_98,owe_99}, as it was first obtained by Rellich in 1953 for $p=2$ (even in higher dimension) and published posthumously in~\cite{rel_56}. 

Generalizations to derivatives of arbitrary order for $p=2$ were subsequently found by Birman~\cite{bir_61}; see also \cite[pp. 83--84]{gla_66}. For $p>1$, the $p$-Birman inequality with sharp constant reads 
\begin{equation}\label{eq:p-birman_cont}
    \int_0^\infty |u^{(\ell)}(x)|^{p}\,\dd x
      \geq \int_0^\infty \rho_{p}^{\mathrm{B}}(x)|u(x)|^p\dd x,
      \quad\text{ with } \rho_{p}^{\mathrm{B}}(x)=\left(\frac{p-1}{p}\right)^{\!p}_{\!\ell}\,\frac{1}{x^{\ell p}},
\end{equation}
where $\ell\in\N$, $u\in W^{\ell,p}(\R_+)$ with $u(0_{+})=u'(0_{+})=\dots=u^{(\ell-1)}(0_{+})=0$, and 
$(a)_\ell :=a(a+1)\dots(a+\ell-1)$ denotes the Pochhammer symbol; see \cite[Thm.~7 and Rem.~10]{sta-wac_prep26} for proofs. Using the notation $\nabla u_{n}=u_{n}-u_{n-1}$ (see \eqref{eq:def_grad_div_Z}), the discrete counterpart of~\eqref{eq:p-birman_cont} takes the form
\begin{equation}
    \label{eq:p-birman_dis}
    \sum_{n=\ell}^{\infty} |\nabla^\ell u_n|^p \geq  \sum_{n=\ell}^{\infty}\rho_{p}^{\mathrm{B}}(n) |u_n|^p,
  \end{equation}
where $\ell\in\N$ and $u\in\ell^{p}(\N_0)$ with $u_0=u_1=\dots =u_{\ell-1}=0$. The discrete $p$-Birman inequality~\eqref{eq:p-birman_dis} was first deduced for $p=2$ by Huang and Ye in~\cite{hua-ye_24} and by the authors for general $p>1$ in~\cite{sta-wac_prep26}.

As with the discrete Hardy inequality, the discrete Rellich and Birman inequalities admit improvements.
To date, only the linear case $p=2$ has been addressed. Initial attempts to improve the discrete Rellich inequality were made in~\cite{ger-kre-sta_25}, followed by~\cite{hua-ye_24}. Subsequently, the authors developed a comprehensive $p=2$ theory in~\cite{sta-wac_26acc}, providing a framework in which the first optimal discrete Rellich and Birman inequalities were established. The objective of the present article is to extend this theory to the nonlinear case of general $p>1$ and to determine the first concrete instances of optimal discrete $p$-Hardy--Rellich--Birman inequalities.
Notably, our results produce new inequalities even for the case $p=2$. For instance, Theorem \ref{thm:optimal_rho_LB_form} with $p=2$ implies that the discrete Birman inequality~\eqref{eq:p-birman_dis} holds with $\rho_{2}^{\mathrm{B}}$ replaced by the strictly bigger weight
\begin{equation}
\rho_{2}^{\mathrm{SW}}(n)=
\left(\frac{1}{2}\right)_{\!\ell}^{\!2}\frac{1}{n(n-1/2)(n-1)\cdots(n-\ell+1/2)},
    \label{eq:2-Birman_dis_opt}
\end{equation}
which is optimal in the sense of Definition \ref{def:opt} below.

\subsection{Organization} The paper is organized as follows. Section~\ref{sec:main_res} introduces essential notation, definitions, and states the main results: abstract sufficient conditions for constructing optimal discrete $p$-Hardy--Rellich--Birman weights (Theorems~\ref{thm:p-Hardy_l}--\ref{thm:optimal_p-Hardy}) and a concrete optimal discrete $p$-Hardy--Rellich--Birman inequality (Theorem~\ref{thm:optimal_rho_LB_form}). The corresponding proofs are provided in Sections~\ref{sec:proofs_abstract} and~\ref{sec:proofs_concrete}. Lastly, Section~\ref{sec:add} extends known results to an alternative discrete $p$-Hardy--Rellich--Birman weight, proving its optimality in specific cases and conjecturing the extension to the full parameter range.

\section{Main results}\label{sec:main_res}

\subsection{Notation}
We use the notation $\Z$, $\N$, and $\N_0$ for the sets of integers, positive integers, and non-negative integers, respectively.
If not stated otherwise, \textbf{we assume throughout this article} that 
\begin{equation}
\boxed{\, \ell\in\N \quad\text{ and }\quad p>1. \,}
\label{eq:assum_p_and_ell}
\end{equation}
Furthermore, $q>1$ always denotes the H{\" o}lder conjugate to $p$, i.e.,
\[
\frac{1}{p}+\frac{1}{q}=1.
\]

We denote by $C(\Z)$ the space of complex sequences indexed by $\Z$, and by $C_{0}(\Z)$ its subspace of finitely supported sequences. We introduce the \textit{discrete gradient} and \textit{discrete divergence} operators acting on $C(\Z)$ as
\begin{equation} 
    \label{eq:def_grad_div_Z}
    (\grad u)_n := u_n - u_{n-1}
    \quad\text{ and }\quad
    (\ddiv u)_n := u_{n+1} - u_n
\end{equation}
for $n \in \Z$. The discrete derivatives $\grad$ and $\ddiv$ commute on $C(\Z)$ and their composition is the discrete Laplacian $\Delta$ on $C(\Z)$, 
\[
(\Delta u)_n:=u_{n-1}-2u_n+u_{n+1},
\]
for $n\in\Z$. For brevity, we omit the parentheses in expressions like $(\grad u)_n$, $(\ddiv u)_n$, $(\Delta u)_n$, etc., and write simply $\grad u_n$, $\ddiv u_n$, $\Delta u_n$, etc.

Although we do not work, in fact, with doubly-infinite sequences, it turns out to be notationally advantageous to formally extend semi-infinite sequences by zeros and therefore we define the space
\[
    H^{\ell}:= \{u \in C(\Z) \mid u_{n} = 0 \text{ for all } n < \ell\}
\]
and its subspaces
\begin{equation*}
    \H^{\ell,p} := H^{\ell} \cap \ell^p(\Z)
    \quad\text{ and }\quad
    \H_0^\ell  := H^\ell \cap C_0(\Z),
\end{equation*}
endowed with the $\ell^p$-norm. Clearly, $\H^{\ell,p}$ can be identified with the $\ell^{p}$-space of sequences indexed by $\ell,\ell+1,\ell+2,\dots$.

The classical Hardy inequality~\eqref{eq:2-Hardy_dis} can be rephrased as the inequality $-\Delta\geq w^{\mathrm{H}}$ on $\H^{1,2}\simeq\ell^{2}(\N)$ understood in the sense of quadratic forms. The $\ell$-th power of $-\Delta$ is the $\ell$-fold composition of $-\Delta$, and $(-\Delta)^{\ell}=(-1)^{\ell}\ddiv^{\ell}\grad^{\ell}$. Then the Birman inequality~\eqref{eq:p-birman_dis} with $p=2$ can be written as $(-\Delta)^{\ell}\geq\rho_{2}^{\mathrm{B}}$ on $\H^{\ell,2}$ in the sense of quadratic forms.

If $p\neq2$, the $p$-Birman inequality~\eqref{eq:p-birman_dis} cannot be interpreted in the sense of quadratic forms but we can express it similarly by polarizing the $\ell^{p}$-norm of $\grad^\ell u$, for $u \in \H^{\ell}_{0}$, as
\begin{equation*}
    \sum_{n=\ell}^{\infty} \bigl|\grad^\ell u_n\bigr|^p
        = \sum_{n=\ell}^{\infty} \overline{u_n}\,(-\lap_p^{(\ell)}u)_n,
\end{equation*}
from which we infer the definition of the \textbf{non-linear} discrete \emph{$p$-Laplacian of order $\ell$}, 
\begin{equation}
    \label{eq:def_p-lap}
    -\lap_p^{(\ell)}u
        := (-1)^\ell \ddiv^\ell\bigl(\grad^\ell u\bigr)^{\brac{p-1}},
\end{equation}
where $\nu^{\brac{a}}:=\nu|\nu|^{a-1}$ for any $a>0$ and $\nu\in\C$, with the convention $0^{\brac{a}}:=0$. Then~\eqref{eq:p-birman_dis} can be formulated as 
\[
\langle u,-\lap_p^{(\ell)}u\rangle\geq\langle u,\rho_{p}^{\mathrm{B}}u\rangle
\]
for $u\in\H_{0}^{\ell}$, where $\langle\,\cdot\, , \,\cdot\,\rangle$ denotes the Euclidean inner product anti-linear in the first argument and linear in the second one.

We also apply the general convention that the $0$-th power of a linear difference operator is the identity operator. In particular, if applied on the right-hand side of~\eqref{eq:def_p-lap}, we obtain
\[
-\lap_p^{(0)}u:=(u)^{\brac{p-1}}.
\]

\subsection{Abstract optimal discrete \texorpdfstring{$p$}{p}-Hardy--Rellich--Birman inequalities}

We survey our first abstract results which encapsulate sufficient conditions imposed on a parameter sequence to give rise to optimal discrete $p$-Hardy--Rellich--Birman inequality. We start with definitions of key notions: the discrete $p$-Hardy--Rellich--Birman weight and its optimality.

\begin{definition}
A non-negative sequence $\rho=\{\rho_n\}_{n=\ell}^\infty$ is called a \emph{discrete $p$-Hardy} ($\ell=1$), \emph{$p$-Rellich} ($\ell=2$), or \emph{$p$-Birman} ($\ell\geq3$) \emph{weight} if and only if the inequality
    \begin{equation} 
        \label{eq:p-hrb_ineq_gener}
        \sum_{n=\ell}^{\infty} \bigl|\grad^\ell u_n\bigr|^p \geq \sum_{n=\ell}^{\infty} \rho_n|u_n|^p
    \end{equation}
holds for all $u\in \H_0^\ell$. In cases when $\ell$ is a general, further non-specified, positive integer, we refer to $\rho$ as the \emph{discrete $p$-Hardy--Rellich--Birman weight}.
\end{definition}

\begin{definition}\label{def:opt}
The $p$-Hardy--Rellich--Birman weight $\rho$ is said to be \textit{optimal} if it exhibits the following three properties:
    \begin{enumerate}[(i)]
        \item 
            \textit{Criticality}: The weight $\rho$ is called \textit{critical} if for any discrete $p$-Hardy--Rellich--Birman weight $\tilde{\rho}$, such that $\tilde{\rho}_n\geq\rho_n$ for all $n\geq \ell$, it follows that $\tilde{\rho}=\rho$.
        \item
            \textit{Non-attainability}: The weight $\rho$ is called \textit{non-attainable} if, whenever the equality in \eqref{eq:p-hrb_ineq_gener} is attained for $u\in H^\ell$ such that the right-hand side of~\eqref{eq:p-hrb_ineq_gener} is finite, then $u\equiv0$.
        \item 
            \textit{Optimality near infinity}: The weight $\rho$ is called \textit{optimal near infinity} if for any $m\ge \ell$ and $\varepsilon>0$ there exists $u\in \H_0^m$ such that 
        \begin{equation*}
            \sum_{n=\ell}^{\infty}\bigl|\grad^\ell u_n\bigr|^p < (1+\varepsilon) \sum_{n=\ell}^{\infty} \rho_n|u_n|^p.
        \end{equation*}
    \end{enumerate}
Additionally, we define a weaker property than (ii): The weight $\rho$ is said to be \textit{non-attainable in} $\H_0^\ell$ if, whenever equality in \eqref{eq:p-hrb_ineq_gener} is attained for $u \in \H_0^\ell$, then $u \equiv 0$.
\end{definition}

\begin{remark}\label{rem:opt}
The definition of optimality is motivated by a definition used in~\cite{kel-pin-pog_cmp18}, but it is not identical. We complement the three optimality properties with explanatory remarks:
\begin{enumerate}[(i)]
        \item 
            The criticality of the weight $\rho$ means that the inequality~\eqref{eq:p-hrb_ineq_gener} does not hold (for all $u\in\H_{0}^{\ell}$) with $\rho$ replaced by a pointwise bigger weight, i.e. the inequality~\eqref{eq:p-hrb_ineq_gener} cannot be improved any further.
        \item 
            Instead of non-attainability, the authors of \cite{kel-pin-pog_cmp18, kel-pin-pog_20} use the notion \emph{null-criticality} which extends the criticality. Assuming $p=2$ and critical $\rho$, it is proved in \cite{kel-pin-pog_cmp18} the existence of a unique (up to a positive multiplicative constant) non-negative non-trivial solution $u$ of the equation $-\Delta u=\rho u$, called the \emph{Agmon ground state}. The critical weight $\rho$ is then called null-critical, if the Agmon ground state $u$ does not belong to the weighted $\ell^{2}$-space with the weight $\rho$, i.e. the right-hand side of~\eqref{eq:p-hrb_ineq_gener} is infinite. The existence of the Agmon ground state was later proven in \cite{fis_23} also in the setting of real $\ell^{p}$-spaces. We prefer the non-attainability as it refers directly to the inequality~\eqref{eq:p-hrb_ineq_gener} itself and avoids the necessity of proving existence of the Agmon ground state. Since the criticality together with the non-attainability imply null-criticality, they constitute a~stronger condition.
        \item 
            The optimality near infinity means that the constant $1$ appearing on the right-hand side of \eqref{eq:p-hrb_ineq_gener} cannot be further improved, even if we restrict ourselves only to the sequences from the space $\H_0^m\subset\H_0^\ell$ for arbitrarily large $m\ge\ell$. It is equivalent to the equality
            \begin{equation}
                \label{eq:opt_near_inf}
                \inf_{u\in\H_0^m\setminus\{0\}}\frac{\sum_{n=\ell}^{\infty}|\grad^\ell u_n|^p}{\sum_{n=\ell}^{\infty} \rho_n|u_n|^p}=1
                \quad\text{ for all }
                m \ge\ell.
            \end{equation}
            Finally, we remark that it is shown in \cite[Thm. 2.6]{fis_24} that for certain discrete $p$-Hardy weights, null-criticality implies optimality near infinity.
\end{enumerate}
\end{remark}

Our starting point is a weighted generalization of an abstract first-order $p$-Hardy inequality analogous to \cite[Thm. 3.1]{fis_23}, where it is formulated in an abstract graph setting for real-valued sequences. The proof of the weighted generalization is reminiscent of the one given in \cite{fis_23} or \cite[Prop.~4]{fis-kel-pog_23}, but for the reader's convenience and completeness, it will be revisited in Section \ref{sec:proofs_abstract}. The inequality is potentially of independent interest and was utilized, without the concrete form of the remainder terms, in~\cite{sta-wac_prep26}.

In the next theorem and below, we use the following notation: for two real numbers (or functions) $\alpha,\beta$, we write $\alpha \lesssim \beta$ if and only if there exists a positive constant $C>0$ such that $\alpha \le C \beta$, where the constant $C$ may depend only on $p$ and possibly also on $\ell$, the two parameters fixed by assumption~\eqref{eq:assum_p_and_ell}.

\begin{theorem} \label{thm:p-Hardy_1}
    Suppose $V_n\geq 0$ for all $n\geq 1$ and $\g\in H^1$ satisfying $\g_{n+1}\geq\g_n>0$ for all $n\geq 1$. Then for any $u\in\H_0^1$, we have the inequalities
    \begin{equation} 
        \label{eq:thm:p-Hardy_1}
        \sum_{n=1}^{\infty} V_{n+1} \mathcal{R}(q;\g,u)_n 
            \lesssim\sum_{n=1}^{\infty} V_n |\grad u_n|^p + \sum_{n=1}^{\infty} \frac{\ddiv(V(\grad\g)^{p-1})_n}{\g_n^{p-1}} |u_n|^p 
            \lesssim\sum_{n=1}^{\infty} V_{n+1} \mathcal{R}(p;\g,u)_n,
    \end{equation}
    where the remainder terms are given by the formulas
    \begin{equation}
        \label{eq:thm:p-Hardy_1_rem}
        \mathcal{R}(s;\g,u)_n :=
        \begin{cases}
            \left|\sqrt{\frac{\g_{n}}{\g_{n+1}}}u_{n+1}-\sqrt{\frac{\g_{n+1}}{\g_{n}}}u_{n}\right|^p 
              & \text{ if } s\in(1,2], \\
            \left|\sqrt{\frac{\g_{n}}{\g_{n+1}}}u_{n+1}-\sqrt{\frac{\g_{n+1}}{\g_{n}}}u_{n}\right|^2 \left(|\ddiv u_n| + \frac{|u_n|}{\g_n}\ddiv\g_n\right)^{p-2} 
              & \text{ if } s\in(2,\infty),
        \end{cases}
    \end{equation}
    for all $n\ge1$, with the convention $0^{p-2}\equiv0$ if $p\in(1,2]$.
\end{theorem}

\begin{remark} $\,$
    \begin{enumerate}[(i)]
        \item The parameter $s$ in \eqref{eq:thm:p-Hardy_1_rem} serves only to select the correct remainder expression. Since the Hölder conjugates $p,q>1$ satisfy $p>2 \Leftrightarrow q<2$, the remainder terms for the lower and upper bounds in \eqref{eq:thm:p-Hardy_1} simply swap their roles when $p$ transitions from $(1,2)$ to $(2,\infty)$.
        \item In the particular case $p=2$, both implicit constants in \eqref{eq:thm:p-Hardy_1} are equal to $1$, i.e., the double-sided inequality \eqref{eq:thm:p-Hardy_1} is actually an identity. This corresponds to \cite[Thm.~0]{sta-wac_26acc}, which asserts, for all $u\in\H_0^1$, the identity
    \[
        \hskip36pt \sum_{n=1}^{\infty} V_{n} |\grad u_{n}|^{2} + \sum_{n=1}^{\infty}\frac{\ddiv(V\grad \g)_{n}}{\g_n}|u_{n}|^{2} = \sum_{n=1}^{\infty} V_{n+1} \left|\sqrt{\frac{\g_{n}}{\g_{n+1}}}u_{n+1}-\sqrt{\frac{\g_{n+1}}{\g_{n}}}u_{n}\right|^2,
    \]
    provided that $\g_0=0$ and $\g_n>0$ for all $n\ge 1$. In this particular case, neither $V$ is required to be non-negative nor $\g$ to be non-decreasing.
    \end{enumerate}    
\end{remark}

The use of the distinguished font for the sequence $\g$ from Theorem~\ref{thm:p-Hardy_1} highlights its central role in our abstract construction. The sequence $\g$, which we refer to as the \emph{parameter sequence}, enters the sufficient conditions of the next theorems, Theorems~\ref{thm:p-Hardy_l}--\ref{thm:optimal_p-Hardy}, as the sole input.

The presence of weight $V$ in \eqref{eq:thm:p-Hardy_1} enables an iterative procedure that yields analogous inequalities for higher-order difference operators. The key idea behind the derivation of such inequalities is a suitable selection of the parameter sequence $\g$ in each step of the iteration. In this way, we obtain convenient representations of the weights $V$ in terms of $\g$. The necessary non-negativity of $V$ induces additional higher-order monotonicity assumptions imposed on $\g$.

\begin{theorem} \label{thm:p-Hardy_l}
    Suppose that $\g\in H^\ell$ satisfies the assumptions
    \begin{align}
        \grad^k\g_n & >0 \quad\text{ for all } n\ge\ell \text{ and } k\in\{0,\dots,\ell\},  \label{assump:1} \tag{A1} \\
        -\lap_p^{(\ell-k)}\grad^{k}\g_{n} &\ge 0 \quad\text{ for all } n \ge \ell \text{ and } k\in\{1,\dots,\ell-1\}. \label{assump:2} \tag{A2}
    \end{align}
    Then for all $u\in\H_0^\ell$, we have the inequalities
    \begin{equation} 
        \label{eq:thm:p-Hardy_l}
        \sum_{k=1}^{\ell}R_{k}^{(\ell)}(q;\g,u)
          \lesssim \sum_{n=\ell}^{\infty} \bigl|\grad^\ell u_n\bigr|^p - \sum_{n=\ell}^{\infty} \frac{-\lap_p^{(\ell)}\g_n}{\g_n^{p-1}}|u_n|^p 
          \lesssim \sum_{k=1}^{\ell}R_{k}^{(\ell)}(p;\g,u),
    \end{equation}
    where the remainders are given by the formula
    \begin{equation}
        \label{eq:thm:p-Hardy_l_rem}
        R_k^{(\ell)}(s;\g,u) := \sum_{n=\ell}^{\infty} \frac{-\lap_p^{(\ell-k)}\grad^{k}\g_{n+1}}{(\grad^{k}\g_{n+1})^{p-1}}\, \mathcal{R}(s;\grad^{k-1}\g,\grad^{k-1} u)_n
    \end{equation}
    with $\mathcal{R}(s;\g,u)$ defined in \eqref{eq:thm:p-Hardy_1_rem}.
\end{theorem}

If we further assume that assumption~\eqref{assump:2} holds also for $k=0$, i.e., $-\lap_p^{(\ell)}\g_n \ge 0$ for all $n \ge \ell$, then $\rho(\g):=-\lap_p^{(\ell)}\g/\g^{p-1}$ defines a non-negative sequence in $H^\ell$. Due to the non-negativity of the left-hand side of~\eqref{eq:thm:p-Hardy_l}, $\rho(\g)$ constitutes a $p$-Hardy--Rellich--Birman weight. 
Moreover, under these conditions, we prove that the weight $\rho(\g)$ is non-attainable in $\H_0^\ell$; see Definition~\ref{def:opt}. 

\begin{theorem}
    \label{thm:p-Hardy_weight}
    Suppose that $\g\in H^\ell$ satisfies assumptions \eqref{assump:1}, \eqref{assump:2}, and, in addition,
    \begin{equation} \label{assumpp:3} \tag{A3}
        -\lap_p^{(\ell)}\g_n\ge0 \quad\text{ for all } n\ge\ell.
    \end{equation}
    Then $\rho(\g)=-\lap_p^{(\ell)}\g/\g^{p-1}$ is a discrete $p$-Hardy--Rellich--Birman weight, which is non-attainable in $\H_0^\ell$.
\end{theorem}

By imposing additional requirements on the asymptotic behavior of $\g_n$, for $n\to\infty$, and strict positivity in assumption \eqref{assumpp:3}, we obtain sufficient conditions for the optimality of the weight $\rho(\g)$. We regard the next theorem as our first main result.

\begin{theorem} \label{thm:optimal_p-Hardy}
    Suppose that $\g\in H^\ell$ satisfies assumptions \eqref{assump:1}, \eqref{assump:2}, and \eqref{assumpp:3} with strict positivity, i.e., $-\lap_p^{(\ell)}\g_n > 0$ for all $n \geq \ell$. In addition, suppose that $\g$ admits the asymptotic expansion
    \begin{equation} \label{assump:4}
        \g_{n} = \sum_{j=0}^{2\ell}\alpha_{j}n^{\ell-j-1/p} + \mathcal{O}(n^{-\ell-1-1/p}),
        \quad\text{ as }n\to\infty,
        \text{ for some } \alpha_{j}\in\R  \text{ with } \alpha_{0}\neq0.
    \tag{A4}
    \end{equation}
    Then $\rho(\g)=-\lap_p^{(\ell)}\g/\g^{p-1}$ is an optimal discrete $p$-Hardy--Rellich--Birman weight.
\end{theorem}

The three Theorems \ref{thm:p-Hardy_l}--\ref{thm:optimal_p-Hardy} are subsequently proven in Section~\ref{sec:proofs_abstract}.

\begin{remark}\label{rem:(A1-3)_impl_weak-opt}
    The strict positivity in assumption \eqref{assumpp:3} is required solely for the non-attainability of the weight $\rho(\g)$, whereas the properties of criticality, optimality near infinity, as well as the weaker non-attainability in $\H_{0}^{\ell}$ are already ensured by assumptions \eqref{assump:1}--\eqref{assump:4}.
\end{remark}

\begin{remark}
    Without going into details, we note that the asymptotic expansion in \eqref{assump:4} can be replaced by a more general one depending on a parameter $s\in(0,1)$. The admissible range of $s$ then determines which of the optimality properties are satisfied by the corresponding weight $\rho(\g)$. Specifically, for $\g\in H^\ell$ subject to assumptions \eqref{assump:1}, \eqref{assump:2}, \eqref{assumpp:3} with strict positivity, and, in addition,
    \[
        \g_{n} = \sum_{j=0}^{2\ell}\alpha_{j}n^{\ell-j-s} + \mathcal{O}(n^{-\ell-1-s})
        \quad\text{ as }n\to\infty,
        \text{ for some } \alpha_{j}\in\R  \text{ with } \alpha_{0}\neq0,
    \]
    the discrete $p$-Hardy--Rellich--Birman weight $\rho(\g)$ is critical if $s\in[1/p,1)$, non-attainable if $s\in(0,1/p]$, and optimal near infinity if $s=1/p$. Since the proof follows almost verbatim the arguments in Section \ref{sec:proofs_abstract}, it is omitted; see also \cite[Thm.~3]{sta-wac_26acc} for an analogous result in the linear case $p=2$.
\end{remark}

\subsection{Concrete optimal discrete \texorpdfstring{$p$}{p}-Hardy--Rellich--Birman inequalities}

By applying Theorem \ref{thm:optimal_p-Hardy}, we can construct concrete discrete $p$-Hardy--Rellich--Birman inequalities with optimal weights, if we find a parameter sequence $\g$ that satisfies all the assumptions. We set
\begin{equation}\label{eq:def:g}
\g_n^{(\ell,p)}:=\frac{\Gamma(n+1/q)}{\Gamma(n-\ell+1)}
            \quad\text{ for all }
            n\in\Z,
\end{equation}
where $\Gamma$ is the Euler Gamma function. Since the reciprocal Gamma function $1/\Gamma$ is entire and vanishes at non-positive integers, it follows that $\g^{(\ell,p)}\in H^{\ell}$. We prove in Section \ref{sec:proofs_concrete} that this sequence yields an optimal weight, which improves upon the classical $p$-Birman weight $\rho^{\mathrm{B}}_p$, see~\eqref{eq:p-birman_cont}--\eqref{eq:p-birman_dis}.

\begin{theorem}\label{thm:optimal_rho_LB_form}
    The sequence
    \begin{equation}
    \label{eq:optimal_rho_LB_form}
    \rho_n^{(\ell,p)}
          :=\frac{-\lap_p^{(\ell)}\g_n^{(\ell,p)}}{\bigl(\g_n^{(\ell,p)}\bigr)^{p-1}}
          = \left(\frac{1}{q}\right)_{\!\ell}^{\!p-1} \left[\frac{\Gamma(n-\ell+1)}{\Gamma(n+1/q)}\right]^{p-1}\sum_{j=0}^{\ell}(-1)^{j}\binom{\ell}{j}\left[\frac{\Gamma(n-\ell+j+1/q)}{\Gamma(n-\ell+j+1)}\right]^{p-1},
\end{equation}
is an optimal discrete $p$-Hardy--Rellich--Birman weight satisfying
\begin{equation}\label{eq:ineq_improve_birman}
    \rho_{n}^{(\ell,p)}>\left(\frac{1}{q}\right)_{\!\ell}^{\!p}\frac{1}{n^{\ell p}}
\end{equation}
for all $n\geq\ell$.
\end{theorem}

\begin{remark}
    The proof of Theorem~\ref{thm:optimal_rho_LB_form} actually provides a tighter inequality than \eqref{eq:ineq_improve_birman}, namely
    \[
    \rho_{n}^{(\ell,p)}>\left(\frac{1}{q}\right)_{\!\ell}^{\!p}\frac{1}{n^{\ell}\left[n(n-1)\dots(n-\ell+1)\right]^{p-1}} \quad\text{ for all } n\geq\ell.
    \]
\end{remark}

\begin{remark}
    If $\ell=1$ or $p=2$, the expression in~\eqref{eq:optimal_rho_LB_form} simplifies. Specifically, one easily finds that 
    \begin{equation}\label{eq:rho_1,p}
    \rho_{n}^{(1,p)}=\left(\frac{1}{q}\right)^{\!p-1}\left(\frac{1}{(n-1/p)^{p-1}}-\frac{1}{n^{p-1}}\right), \quad n\geq 1.\end{equation}
    and, with the aid of the Chu--Vandermonde identity, see~\eqref{eq:chu-van_id} below, also that
    \begin{equation}\label{eq:rho_l,2}
    \rho_{n}^{(\ell,2)}=\left(\frac{1}{2}\right)_{\!\ell}^{\!2}\frac{1}{(n-\ell-1/2)_{\ell}(n-\ell+1)_{\ell}}, \quad n\geq\ell.
    \end{equation}
    These are different forms of the weights \eqref{eq:p-Hardy_dis_opt} and \eqref{eq:2-Birman_dis_opt} from the introduction. We stress that these optimal weights yield new inequalities complementing those discovered in earlier works \cite{fis-kel-pog_23, sta-wac_26acc} for $\ell=1$ or $p=2$.
\end{remark}

\begin{remark}\label{rem:conj_posit_coeff}
    With the aid of the Stirling formula, one may expand~\eqref{eq:optimal_rho_LB_form} for $n$ large. The leading term reveals the classical $p$-Birman weight,
    \[
    \rho_{n}^{(\ell,p)}=\left(\frac{1}{q}\right)_{\!\ell}^{\!p}\frac{1}{n^{\ell p}}\left[1+\mathcal{O}\left(\frac{1}{n}\right)\right], \quad n\to\infty.
    \]
    Beyond this asymptotic behavior, we conjecture that $\rho_{n}^{(\ell,p)}$ admits a convergent series expansion
    \begin{equation}
        \rho^{(\ell,p)}_n = \left(\frac{1}{q}\right)_{\!\ell}^{\!p}\frac{1}{n^{\ell p}}\left(1+\sum_{k=1}^\infty \frac{A_k^{(\ell,p)}}{n^{k}}\right) \quad\text{ for all }n\ge\ell,
    \label{eq:rho_ser_exp_pos}
    \end{equation}
    with the coefficients $A_k^{(\ell,p)}$ \textbf{positive} for all $k\in\N$. This positivity represents a stronger property than \eqref{eq:ineq_improve_birman} and can be readily established for the cases $\ell=1$ and $p=2$ using the simplified expressions in \eqref{eq:rho_1,p} and \eqref{eq:rho_l,2}. Specifically, one computes
    \[
     A_k^{(1,p)}=\frac{(p)_{k}}{p^k}\frac{1}{(k+1)!}
     \quad\text{ and }\quad
     A_k^{(\ell,2)}=\frac{S(2\ell-1+k,2\ell-1)}{2^k},
    \]
    where $S(n,k)$ are the (positive) Stirling numbers of the second kind, see \cite[Eqs. (26.8.5) and (26.8.11)]{dlmf}. Such positivity is a useful property as, by Theorem~\ref{thm:optimal_rho_LB_form}, truncations of the series for $\rho^{(\ell,p)}_n$ would yield explicit $p$-Hardy--Rellich--Birman weights gradually improving upon the classical $p$-Birman weight. A similar positivity question for the $p$-Hardy weight~\eqref{eq:weight_fkp} was raised in~\cite{fis-kel-pog_23}, and resolved recently in~\cite{sta-wac_26}.
\end{remark}

\begin{remark}
The second term of the asymptotic expansion for optimal Hardy weights is sometimes a point of comparison, as the leading term matches the classical weight. We remark that the weight
    \[
     \rho^{(1,p)}_{n}=\left(\frac{1}{q}\right)^{\!p}\frac{1}{n^{p}}\left[1+\frac{1}{2n}+\mathcal{O}\left(\frac{1}{n^{2}}\right)\right], \quad n\to\infty,
    \]
    has asymptotically heavier tail than~\eqref{eq:weight_fkp} and, for $p=2$, than the optimal Hardy weights from~\cite[Eq.~(5.2)]{das-fue_26} and \cite[Rem.~12]{sta-wac_26acc}. A similar observation holds for the discrete Rellich weight
    \[
     \rho^{(2,2)}_{n}=\frac{9}{16n^{4}}\left[1+\frac{3}{n}+\mathcal{O}\left(\frac{1}{n^{2}}\right)\right], \quad n\to\infty,
    \]
    compared to the Rellich weights from~\cite[Rem.~6e)]{sta-wac_26acc} and \cite[Eq.~(7.5)]{hua-ye_24}. 
\end{remark}

As an alternative to~\eqref{eq:def:g}, one may consider a different parameter sequence that aligns more closely with the continuous setting and with previous research \cite{fis-kel-pog_23, sta-wac_26acc} focused on the cases $\ell=1$ or $p=2$. This approach suggests replacing \eqref{eq:def:g} with the definition
\begin{equation}\label{eq:def:g_tilde}
            \tilde{\g}_n^{(\ell,p)} := n^{1-1/p} \prod_{j=1}^{\ell-1} (n-j)
            \quad\text{ for }n\in\N_{0},
        \end{equation}
setting $\tilde{\g}_n:=0$ for $n<0$. Then the corresponding weight 
\begin{equation}\label{eq:def:rho_tilde}
\tilde{\rho}_{n}^{(\ell,p)}:=\frac{-\lap_p^{(\ell)}\tilde{\g}_n^{(\ell,p)}}{\bigl(\tilde{\g}_n^{(\ell,p)}\bigr)^{p-1}}
\end{equation}
coincides with the Fischer--Keller--Pogorzelski optimal $p$-Hardy weight~\eqref{eq:weight_fkp} for $\ell=1$, and with the optimal discrete Birman weight from \cite[Thm.~5]{sta-wac_26acc} for $p=2$.

However, for general $p>1$ and $\ell\in\N$, we were unable to verify that $\tilde{\g}^{(\ell,p)}$ meets all the assumptions of Theorem~\ref{thm:optimal_p-Hardy}. The primary advantage of our original parameter sequence \eqref{eq:def:g} lies in its more favorable behavior under discrete differentiation compared to \eqref{eq:def:g_tilde}, see Lemma~\ref{lem:div_g}. We conclude this article with partial results extending the known cases, along with related conjectures, in Section \ref{sec:add}.

\subsection{Related literature}

The literature on Hardy inequalities is extraordinarily extensive, making it impossible to cover all related results. Therefore, below, we explicitly mention only closely related results and contributions specifically addressing discrete inequalities of Hardy type.

\vskip8pt
\noindent\textbf{Work of F.~Fischer:}
In his doctoral dissertation~\cite{fis_diss24} (subsequently published in \cite{fis_23,fis_24, fis_25}), Fischer developed a comprehensive framework for a general theory of optimal discrete $p$-Hardy inequalities on graphs, extending previous works \cite{kel-pin-pog_cmp18, kel-pin-pog_20} to the non-linear case of general $p>1$. Although our scope differs---as we restrict our attention to the simple case of a discrete half-line graph but investigate \textbf{higher-order Hardy-type inequalities} of Rellich and Birman type with an emphasis on concrete weights---our works intersect in several aspects.

First, our definition of optimality is inspired by~\cite{fis_diss24} and related previous works~\cite{dev-fra-pin_14, kel-pin-pog_cmp18}, although we adopt a stronger notion by replacing null-criticality with non-attainability (as detailed in Remark~\ref{rem:opt}).
Second, we build upon the non-trivial fundamental inequalities from~\cite[Lemma 3.8]{fis_23}, extending them to the complex domain in Lemma~\ref{lem:ineq}. As our motivation is disconnected from probability theory, complex sequence spaces are assumed in accordance with standard functional analysis conventions. In this respect, the Hardy case ($\ell=1$ and $p>1$) of Theorem~\ref{thm:tilde_g_opt} extends the main result of~\cite{fis-kel-pog_23} to \textbf{complex sequences and stricter optimality}. Finally, a core idea of the supersolution construction technique applied in~\cite{fis_diss24} (and previously used also in~\cite{dev-pin_16, kel-pin-pog_cmp18, ver_23}) is implicitly present in our proof of Theorem~\ref{thm:optimal_p-Hardy}.

\vskip8pt
\noindent\textbf{Work of the authors for $p=2$:} This paper continues the investigation initiated in \cite{sta-wac_26acc} for the case $p=2$. While the discrete Laplacian~\eqref{eq:def_p-lap} is linear for $p=2$ and inequality~\eqref{eq:p-hrb_ineq_gener} can be understood via quadratic forms on a Hilbert space, the current work for general $p>1$ represents a transition to a \textbf{non-linear} setting.
We highlight here the main differences and challenges of this transition.

First, the crucial upper and lower bounds from~\eqref{eq:thm:p-Hardy_l} hold as identities when $p=2$, where, in addition, the error terms can be interpreted as non-negative linear operators; see \cite[Thm.~1]{sta-wac_26}. This identity provides a more direct control of the key central term and allows for an extension lemma (see \cite[Lemma~10]{sta-wac_26acc}) that is unavailable in the non-linear case $p>1$. Otherwise, the proofs of Theorems~\ref{thm:p-Hardy_l}--\ref{thm:optimal_p-Hardy} remain methodologically similar to the linear case.

Another essential innovation of the present work is the explicit choice of the parameter sequence~\eqref{eq:def:g} that meets all the assumptions of Theorem~\ref{thm:optimal_p-Hardy}, which was discovered by an experienced guess. This choice would have been scarcely possible without insights from the linear case, even though the $p=2$ version of \eqref{eq:def:g} does not appear in \cite{sta-wac_26acc}. Furthermore, the proof of Theorem~\ref{thm:optimal_rho_LB_form} relies on the properties of (logarithmically) completely monotone sequences and is incomparable with its much simpler counterpart for $p=2$ or the similar result from \cite[Thm.~5]{sta-wac_26acc}. The new connection between complete monotonicity and optimal weights of Hardy type (whose even stronger form was crucial in our work~\cite{sta-wac_26}) is not yet fully understood and might represent a promising avenue for future research.

\vskip8pt
\noindent\textbf{Other relevant works:}
Diverse definitions exist for powers of the discrete Laplacian, even within the linear setting $p=2$ (see \cite[Sec.~1.4]{sta-wac_26acc}). The definition used here preserves the Toeplitz structure of the discrete Laplacian on $\N$, which makes it a close analogy to its continuous counterpart. Alternative definitions for the fractional Laplacian $(-\Delta)^{\alpha}$ on $\N$, with possibly non-integral $\alpha>0$, have been explored and Hardy-type inequalities established in \cite{ger-kre-sta_25b}. Their optimal improvements, for $0<\alpha<1$, were found recently in \cite{das-fue_26}. On the full line $\Z$, optimal discrete fractional Hardy inequalities were found in \cite{cia-ron_18}, with their optimality later verified in \cite{kel-nit_23}.

Other works devoted to discrete Hardy and related inequalities include \cite{das-etal_25, das-man-pau_25, gup_24} in the one-dimensional setting, and \cite{gup_23, gup_24b, guz-kap-lap_22, kap-lap_16, kel-lem_23} in higher dimensions. To our best knowledge, all these works assume the linear framework ($p=2$).

\section{Proofs for the abstract framework}\label{sec:proofs_abstract}

In this section, we subsequently prove Theorems~\ref{thm:p-Hardy_1} and \ref{thm:p-Hardy_l}--\ref{thm:optimal_p-Hardy}.

\subsection{Proof of Theorem~\ref{thm:p-Hardy_1}}
Before proving Theorem~\ref{thm:p-Hardy_1}, we deduce auxiliary but fundamental inequalities in the next lemma.
The lemma can be viewed as an extension of the non-trivial result \cite[Lemma 3.8]{fis_23}, established by Fischer in the real domain, to the complex domain.

\begin{lemma}
    \label{lem:ineq}
    Let $t\in[0,1]$ and $z\in\C$. If $p\in(1,2]$, then we have the inequalities
    \begin{equation}
        t|z-1|^2(|z-t|+1-t)^{p-2} 
           \lesssim |z-t|^p - (1-t)^{p-1}(|z|^p-t) 
           \lesssim t^{p/2}|z-1|^p. \label{eq:lem:ineq_p_in_(1,2]}
    \end{equation}
    Conversely, if $p>2$, both inequalities in \eqref{eq:lem:ineq_p_in_(1,2]} are reversed.
    (If $z=t=1$, the left-hand side of \eqref{eq:lem:ineq_p_in_(1,2]} is to be understood as $0$.)
\end{lemma}
\begin{proof}
    The proof is divided into two parts, in which we establish
    \begin{enumerate}[a)]
        \item \label{enum:pr:lem:a}
            the lower bound for $p\in(1,2]$, and, respectively, the upper bound for $p>2$,
        \item \label{enum:pr:lem:b}
            the upper bound for $p\in(1,2]$, and, respectively, the lower bound for $p>2$,
    \end{enumerate}
    for the middle term in \eqref{eq:lem:ineq_p_in_(1,2]}. In both cases, the strategy is to extend the corresponding real-variable estimates, originally formulated in \cite[Lemma~3.8]{fis_23}, to the complex setting. However, the methods employed in \ref{enum:pr:lem:a}) and \ref{enum:pr:lem:b}) differ. Therefore they are treated separately.

    \ref{enum:pr:lem:a}) Suppose $p\in(1,2]$. By \cite[Eq.~(3.8)]{fis_23}, there exists $C_p>0$, such that the inequality
    \begin{equation}
        \label{eq:lem:pr:fisher:a}
        C_p\, t|x-1|^2(|x-t|+1-t)^{p-2} \leq |x-t|^p - (1-t)^{p-1}(|x|^p - t)
    \end{equation}
    holds for all $t\in[0,1]$ and $x\in\R$. For $z\in\C$, write $z=r\ee^{\ii\theta}$, where $r\geq0$ and $\theta\in[0,2\pi)$. Next, for $u:=\cos\theta\in[-1,1]$, denote
    \[
        A(u):=|z-t|=\sqrt{r^2-2rtu+t^2}
        \quad\text{ and }\quad
        B(u):=|z-1|=\sqrt{r^2-2ru+1}.
    \]
    We will show that the auxiliary function
    \[
        F(u):=A^p(u)-(1-t)^{p-1}(r^p-t)-C_p \,t B^2(u) (A(u)+1-t)^{p-2},
    \]
    defined on the interval $[-1,1]$, is non-increasing, and hence attains its minimum at $u=1$. Since $F(1)\ge0$ by inequality \eqref{eq:lem:pr:fisher:a}, it follows that $F(u)\ge0$ for all $u\in[-1,1]$, which amounts to the lower estimate in \eqref{eq:lem:ineq_p_in_(1,2]}. Indeed, a routine calculation yields
    \begin{equation}
        \label{eq:lem:der}
        \begin{aligned}
            F'&=-prtA^{p-2}+C_prt\left[2(A+1-t)^{p-2}+(p-2)tA^{-1}B^2(A+1-t)^{p-3}\right] \\
              &\le rt\left[2C_p(A+1-t)^{p-2}-pA^{p-2}\right] \le rtA^{p-2}(2 C_p-p)\le0
        \end{aligned}
    \end{equation}
    on $(-1,1)$, where we used the fact that $p\in(1,2]$ twice and, without loss of generality, assumed that $C_p\le p/2$ since the constant $C_p$ in \eqref{eq:lem:pr:fisher:a} can be taken as small as needed.
    
    The case $p>2$ can be treated in an analogous manner. By~\cite[Eq.~(3.8)]{fis_23}, \eqref{eq:lem:pr:fisher:a} holds with the reversed inequality and an arbitrarily large constant $C_p>0$. Similarly, the estimates in~\eqref{eq:lem:der} hold with the reversed inequalities if $p>2$. Therefore the auxiliary function $F$ attains its maximum on $[-1,1]$ at $u=1$ with $F(1)\leq0$. Hence $F(u)\leq0$ for all $u\in[-1,1]$.

    \ref{enum:pr:lem:b}) Suppose $p\in(1,2]$. By \cite[Eq. (3.10)]{fis_23}, there exists $C_p>0$ such that the inequality
    \begin{equation}
        \label{eq:lem:pr:fisher:b}
        |x-t|^p - (1-t)^{p-1}(|x|^p - t)\leq C_p\, t^{p/2}|x-1|^p
    \end{equation}
    holds for all $t\in[0,1]$ and $x\in\R$. We shall also make use of the elementary inequalities 
    \begin{equation}\label{eq:aux_ineq_ab}
        2^{s-1} (a^s + b^s) \leq (a + b)^s \leq a^s + b^s,
    \end{equation}
    which hold for all $a,b\geq 0$ and $s\in(0,1]$.
    Writing $z=r\ee^{\ii\theta}$ and $u=\cos\theta$ as before, we infer from~\eqref{eq:aux_ineq_ab} with $s:=p/2\in(0,1]$ the estimates
    \begin{align}
        |z-t|^p=\left[|r - t|^2 + 2rt(1-u)\right]^{p/2} &\leq |r-t|^p + (2rt)^{p/2}(1-u)^{p/2}, \label{eq:impl:1} \\
        |z-1|^p= \left[|r - 1|^2 + 2r(1-u)\right]^{p/2} &\geq 2^{p/2-1}\left(|r-1|^p+(2r)^{p/2}(1-u)^{p/2}\right).\label{eq:impl:2}
    \end{align}
    Without loss of generality, we may assume that $C_p\ge1$ in~\eqref{eq:lem:pr:fisher:b}. Then estimates \eqref{eq:impl:1}, \eqref{eq:lem:pr:fisher:b}, and \eqref{eq:impl:2} imply
    \begin{align*}
        |z-t|^p-(1-t)^{p-1}(|z|^p-t) 
          &\leq |r-t|^p -(1-t)^{p-1}(r^p-t) + (2rt)^{p/2}(1-u)^{p/2} \\
          &\leq C_p\, t^{p/2} \left(|r-1|^p +(2r)^{p/2}(1-u)^{p/2}\right) \\ 
          &\leq 2^{1-p/2} C_p\, t^{p/2}|z-1|^p,
    \end{align*}
    which is the desired upper bound in~\eqref{eq:lem:ineq_p_in_(1,2]}.

    If $p>2$, once more, all the inequalities \eqref{eq:lem:pr:fisher:b}, \eqref{eq:impl:1}, and \eqref{eq:impl:2} are reversed, and the inverted final estimates hold, provided that $C_p>0$ is chosen sufficiently small. The proof of Lemma \ref{lem:ineq} is complete.
\end{proof}

\begin{proof}[Proof of Theorem~\ref{thm:p-Hardy_1}]
    Suppose $\g\in H^1$ is a non-decreasing sequence such that $\g_n>0$ for all $n\geq 1$ and $u\in\H_0^1$. We will show that for all $n\ge2$, we have the inequalities
\begin{equation} \label{eq:pr:thm:p-Hardy_1:1}
    \mathcal{R}(q;\g,u)_{n-1}
      \lesssim |\grad u_n|^p-\left(\frac{|u_n|^p}{\g_n^{p-1}}-\frac{|u_{n-1}|^p}{\g_{n-1}^{p-1}}\right)(\grad\g_n)^{p-1}
      \lesssim \mathcal{R}(p;\g,u)_{n-1},
\end{equation}
where the remainders are given by the formula~\eqref{eq:thm:p-Hardy_1_rem} and the unspecified constants depend only on $p$. Then, multiplying both sides of~\eqref{eq:pr:thm:p-Hardy_1:1} by the given $V_n\geq0$ and summing over $n$ from $2$ to $\infty$, we obtain inequalities
\[
     \sum_{n=2}^{\infty} V_n|\grad u_n|^p + \sum_{n=1}^{\infty} V_{n+1} (\grad\g_{n+1})^{p-1} \frac{|u_n|^p}{\g_n^{p-1}} - \sum_{n=2}^{\infty} V_n (\grad\g_n)^{p-1} \frac{|u_n|^p}{\g_n^{p-1}}
     \begin{cases}
         \gtrsim\displaystyle\sum_{n=2}^{\infty} V_n \mathcal{R}(q;\g,u)_{n-1}, \\[14pt]
         \lesssim\displaystyle\sum_{n=2}^{\infty} V_n \mathcal{R}(p;\g,u)_{n-1},
     \end{cases}
\]
which are equivalent to
\[
    \sum_{n=2}^{\infty} V_n|\grad u_n|^p + \sum_{n=2}^{\infty}\frac{\ddiv(V(\grad\g)^{p-1})_n}{\g_n^{p-1}} |u_n|^p + V_{2}(\grad\g_{2})^{p-1} \frac{|u_1|^p}{\g_1^{p-1}}
    \begin{cases}
         \gtrsim\displaystyle\sum_{n=1}^{\infty} V_{n+1} \mathcal{R}(q;\g,u)_{n}, \\[14pt]
         \lesssim\displaystyle\sum_{n=1}^{\infty} V_{n+1} \mathcal{R}(p;\g,u)_{n}.
     \end{cases}
\]
Taking also into account that $u_0=\g_0=0$ and $\g_1>0$, we infer the equality
\[
    V_1|\grad u_1|^p = V_1(\grad\g_{1})^{p-1} \frac{|u_1|^p}{\g_1^{p-1}}.
\]
Adding and subtracting this term on the left-hand side of the latter inequality, we establish \eqref{eq:thm:p-Hardy_1}.

To complete the proof of Theorem~\ref{thm:p-Hardy_1}, it remains to verify~\eqref{eq:pr:thm:p-Hardy_1:1}. We shall confine our attention to the case $p \in (1,2]$. The case $p \in (2,\infty)$ is to be treated analogously; it only uses reversed forms of inequalities \eqref{eq:lem:ineq_p_in_(1,2]} and \eqref{eq:pr:thm:p-Hardy_1:2}, \eqref{eq:pr:thm:p-Hardy_1:3} below. For $p\in(1,2]$, inequalities~\eqref{eq:pr:thm:p-Hardy_1:1} read
\begin{equation}
    \label{eq:pr:thm:p-Hardy_1:2}
    |\grad u_n|^p-\left(\frac{|u_n|^p}{\g_n^{p-1}}-\frac{|u_{n-1}|^p}{\g_{n-1}^{p-1}}\right)\!(\grad\g_n)^{p-1}\!
    \begin{cases}
        \gtrsim \!\left|\sqrt{\frac{\g_{n-1}}{\g_{n}}}u_{n}\!-\!\sqrt{\frac{\g_{n}}{\g_{n-1}}}u_{n-1}\right|^2 \!\left(|\grad u_n|\!+\!\frac{|u_{n-1}|}{\g_{n-1}}\grad\g_n\right)^{p-2}\!,\\[10pt]
        \lesssim \!\left|\sqrt{\frac{\g_{n-1}}{\g_{n}}}u_{n}\!-\!\sqrt{\frac{\g_{n}}{\g_{n-1}}}u_{n-1}\right|^p\!.
    \end{cases}
\end{equation}

We first consider the case $u_{n-1}=0$. If also $u_n=0$, then~\eqref{eq:pr:thm:p-Hardy_1:2} holds trivially. If $u_n\neq0$, then~\eqref{eq:pr:thm:p-Hardy_1:2} stems from inequalities
\begin{equation}
    \label{eq:pr:thm:p-Hardy_1:3}
    (p-1)t \le 1 - (1-t)^{p-1} \le t \le t^{p/2},
\end{equation}
which hold for every $t\in[0,1]$, upon substituting $t=\g_{n-1}/\g_n\in(0,1]$, where the last inclusion is ensured by the monotonicity of $\g$.

Next, suppose $u_{n-1}\neq0$. Set temporarily $h_n:= u_n/\g_n$ and substitute for $z=h_n/h_{n-1}\in\C$ and $t=\g_{n-1}/\g_n\in(0,1]$ into~\eqref{eq:lem:ineq_p_in_(1,2]}. It yields the inequalities
\begin{align*}
    \frac{\g_{n-1}}{\g_n} &\left|\frac{h_n}{h_{n-1}}-1\right|^2\left(\left|\frac{h_n}{h_{n-1}}-\frac{\g_{n-1}}{\g_n}\right|+1-\frac{\g_{n-1}}{\g_n}\right)^{p-2} \\
      &\lesssim \left|\frac{h_n}{h_{n-1}}-\frac{\g_{n-1}}{\g_n}\right|^p - \left(1-\frac{\g_{n-1}}{\g_n}\right)^{p-1} \left(\left|\frac{h_n}{h_{n-1}}\right|^p-\frac{\g_{n-1}}{\g_n}\right) \lesssim\left(\frac{\g_{n-1}}{\g_n}\right)^{p/2}\left|\frac{h_n}{h_{n-1}}-1\right|^p.
\end{align*}
Finally, multiplying these inequalities by $\g_n^p|h_{n-1}|^p$ implies~\eqref{eq:pr:thm:p-Hardy_1:2}. The proof of Theorem~\ref{thm:p-Hardy_1} is complete.
\end{proof}

\subsection{Proof of Theorem~\ref{thm:p-Hardy_l}}

We will make use of the \textit{forward shift operator} $\shift$ defined on $C(\Z)$ by
\begin{equation}
    \label{eq:def:shift_mid}
    \shift u_n:= u_{n+1} \quad\text{ for all } n\in\Z,
\end{equation}
and utilize the obvious identities $\ddiv=\shift-\id$, $\grad=\id-\shift^{-1}$, and $\ddiv=\grad\shift=\shift\grad$, whenever needed.

\begin{proof}[Proof of Theorem~\ref{thm:p-Hardy_l}]
The proof proceeds by induction in $\ell\in\N$. The base case $\ell=1$ is precisely Theorem~\ref{thm:p-Hardy_1} with $V\equiv1$, as the assumption~\eqref{assump:1} implies that $\g\in H^1$ satisfies $\g_{n+1}>\g_{n}>0$ for all $n\in\N$. 

Suppose $\ell\ge2$. First notice that
\[
    \sum_{n=\ell}^{\infty}\bigl|\grad^\ell u_n\bigr|^p = \sum_{n=\ell-1}^{\infty}\bigl|\grad^{\ell-1}\ddiv u_n\bigr|^p
\]
for any $u\in\H_0^\ell$. Second, given $\g\in H^\ell$ satisfying assumptions~\eqref{assump:1} and~\eqref{assump:2}, we infer that $\ddiv\g=\shift\grad\g\in H^{\ell-1}$ satisfies these assumptions with $\ell$ replaced by $\ell-1$. Then the induction hypothesis 
applied to $\ddiv u\in\H_0^{\ell-1}$ in place of $u$ and $\ddiv \g$ in place of $\g$ yields
\[
\sum_{n=\ell}^{\infty}\bigl|\grad^\ell u_n\bigr|^p - \sum_{n=\ell-1}^{\infty}\frac{-\lap_p^{(\ell-1)}\ddiv\g_n}{(\ddiv\g_n)^{p-1}}|\ddiv u_n|^p
      \begin{cases}
         \gtrsim\displaystyle\sum_{k=1}^{\ell-1} R_k^{(\ell-1)}(q;\ddiv\g,\ddiv u), \\[14pt]
         \lesssim\displaystyle\sum_{k=1}^{\ell-1} R_k^{(\ell-1)}(p;\ddiv\g,\ddiv u).
     \end{cases}
\]
Recalling definitions of the remainders \eqref{eq:thm:p-Hardy_l_rem} and \eqref{eq:thm:p-Hardy_1_rem}, we observe that, for any $s>1$ and $k\in\{1,\dots,\ell-1\}$, we have $R_k^{(\ell-1)}(s;\ddiv \g,\ddiv u)=R_{k+1}^{(\ell)}(s;\g,u)$, and so we obtain
\begin{equation}
    \label{eq:pr:thm:3:1}
      \sum_{k=2}^{\ell} R_k^{(\ell)}(q;\g,u)\lesssim
      \sum_{n=\ell}^{\infty}\bigl|\grad^\ell u_n\bigr|^p - \sum_{n=\ell-1}^{\infty}\frac{-\lap_p^{(\ell-1)}\ddiv\g_n}{(\ddiv\g_n)^{p-1}}|\ddiv u_n|^p
      \lesssim\displaystyle\sum_{k=2}^{\ell} R_k^{(\ell)}(p;\g,u).
\end{equation}

Next, using the shift operator~\eqref{eq:def:shift_mid}, we shift the summation index in the second sum of the middle term from \eqref{eq:pr:thm:3:1} as
\[
    \sum_{n=\ell-1}^{\infty} \frac{-\lap_p^{(\ell-1)}\ddiv\g_n}{(\ddiv\g_n)^{p-1}}|\ddiv u_n|^p = \sum_{n=1}^{\infty} \frac{\shift^{\ell-1}(-\lap_p^{(\ell-1)}\grad\g)_n}{(\shift^{\ell-1}\grad\g_n)^{p-1}}\bigl|\grad\shift^{\ell-1} u_n\bigr|^p,
\]
and apply Theorem~\ref{thm:p-Hardy_1} with $\shift^{\ell-1}u\in\H_0^{1}$, non-decreasing sequence $\shift^{\ell-1}\g\in H^1$, and the weight
\[
    V_n:=\frac{\shift^{\ell-1}(-\lap_p^{(\ell-1)}\grad\g)_n}{(\shift^{\ell-1}\grad\g_n)^{p-1}},
\]
which is non-negative for all $n\ge1$ by assumption~\eqref{assump:2}. The resulting inequalities read
\[
    \sum_{n=1}^{\infty} V_n \bigl|\grad\shift^{\ell-1} u_n\bigr|^p+\sum_{n=1}^{\infty}\frac{\ddiv(V(\grad\shift^{\ell-1}\g)^{p-1})_n}{(\shift^{\ell-1}\g_n)^{p-1}}\bigl|\shift^{\ell-1} u_n\bigr|^p
    \begin{cases}
        \gtrsim \displaystyle\sum_{n=1}^{\infty}V_{n+1}\mathcal{R}(q;\shift^{\ell-1}\g,\shift^{\ell-1} u)_n, \\[14pt]
        \lesssim \displaystyle\sum_{n=1}^{\infty}V_{n+1}\mathcal{R}(p;\shift^{\ell-1}\g,\shift^{\ell-1} u)_n.
    \end{cases}
\]
Inspecting the prefactor in the second sum, we find
\[
    \ddiv(V(\grad\shift^{\ell-1}\g)^{p-1})
      =\ddiv\left(\frac{\shift^{\ell-1}(-\lap_p^{(\ell-1)}\grad\g)}{(\shift^{\ell-1}\grad\g)^{p-1}}(\shift^{\ell-1}\grad\g)^{p-1}\right)=-\shift^{\ell-1}(-\lap_p^{(\ell)}\g).
\]
Therefore, bearing in mind definition~\eqref{eq:thm:p-Hardy_l_rem}, the last inequalities rewrite as
\[
    R_1^{(\ell)}(q;\g,u)
      \lesssim \sum_{n=\ell-1}^{\infty}\frac{-\lap_p^{(\ell-1)}\ddiv\g_n}{(\ddiv\g_n)^{p-1}}|\ddiv u_n|^p - \sum_{n=\ell}^{\infty}\frac{-\lap_p^{(\ell)}\g_n}{\g_n^{p-1}}|u_n|^p
      \lesssim {R}_1^{(\ell)}(p;\g,u).
\]
Combining these inequalities with \eqref{eq:pr:thm:3:1} implies the desired estimates~\eqref{eq:thm:p-Hardy_l}.
The proof of Theorem~\ref{thm:p-Hardy_l} is complete.
\end{proof}

\subsection{Proof of Theorem~\ref{thm:p-Hardy_weight}}
Assumptions \eqref{assump:1}, \eqref{assump:2} guarantee non-negativity of the remainder terms in~\eqref{eq:thm:p-Hardy_l}, see \eqref{eq:thm:p-Hardy_l_rem}. In particular, the non-negative lower bound implies that inequality~\eqref{eq:p-hrb_ineq_gener} holds for all $u\in\H_0^\ell$ with the weight $\rho(\g)=-\lap_p^{(\ell)}\g/\g^{p-1}$, which is also non-negative by assumption \eqref{assumpp:3}. Consequently, $\rho(\g)$ is the discrete $p$-Hardy--Rellich--Birman weight.

It remains to prove the non-attainability of $\rho(\g)$ in $\H_0^\ell$. Suppose that $u\in\H_0^\ell$ satisfies
\[
    \sum_{n=\ell}^{\infty} \bigl|\grad^\ell u_n\bigr|^p=\sum_{n=\ell}^{\infty} \rho_n(\g)|u_n|^p.
\]
It follows from~\eqref{eq:thm:p-Hardy_l}, bearing the non-negativity of remainder terms $R_k^{(\ell)}(q;\g,u)$ in mind, that  
$R_k^{(\ell)}(q;\g,u) = 0$ for all $k \in \{1,\dots,\ell\}$. In particular, for $k=\ell$, we derive a necessary condition (which is, in fact, sufficient)
\begin{equation}
    \mathcal{R}(q;\grad^{\ell-1}\g,\grad^{\ell-1} u)_n = 0
    \quad\text{ for all }n\ge\ell;
\label{eq:recur_weak-non-att_inproof}
\end{equation}
see definition~\eqref{eq:thm:p-Hardy_l_rem}. If $p \in (1,2)$ (and hence $q>2$), assumption~\eqref{assump:1} implies, in addition, that both $\grad^{\ell-1}\g_n$ and $\grad^\ell\g_n$ are strictly positive for all $n \ge \ell$. Thus, if the second factor on the right-hand side in~\eqref{eq:thm:p-Hardy_1_rem} vanishes, then $\grad^{\ell-1}u_n = \grad^{\ell-1}u_{n+1} = 0$. In any case, regardless of $p<2$ or $p\geq2$, we conclude that the sequence $u\in\H_0^\ell$ satisfies the $\ell$-th order difference equation 
\[
    \sqrt{\frac{\grad^{\ell-1}\g_n}{\grad^{\ell-1}\g_{n+1}}}\grad^{\ell-1}u_{n+1} - \sqrt{\frac{\grad^{\ell-1}\g_{n+1}}{\grad^{\ell-1}\g_{n}}}\grad^{\ell-1}u_{n} = 0
    \quad\text{ for all }n\ge\ell,
\]
with the boundary conditions $u_1=\dots=u_{\ell-1}=0$. This equation has a unique solution up to a multiplicative constant, which is obviously $u=c\,\g$, with a constant $c\in\C$. By assumption~\eqref{assump:1} for $k=0$, $\g_n>0$ for all $n\ge\ell$, therefore the sequence $u$ has a compact support only if $c=0$, which means that $u\equiv0$. The proof of Theorem~\ref{thm:p-Hardy_weight} is complete.
\qed

\subsection{Proof of Theorem~\ref{thm:optimal_p-Hardy}}

The proof is divided into three distinct parts, in which we demonstrate that the weight $\rho(\g) = -\lap_p^{(\ell)}\g / \g^{p-1}$ satisfies the three properties of optimality (see Definition~\ref{def:opt}): 
\begin{center}
    a) criticality; \hskip6pt b) non-attainability; \hskip6pt  c) optimality near infinity. 
\end{center}
The proof relies on the two-sided control~\eqref{eq:thm:p-Hardy_l} and the fact that the parameter sequence $u=\g$ annihilates the remainder terms. While the upper bound is essential for proving the criticality and optimality of $\rho(\g)$ near infinity, the lower bound is crucial for the proof of the non-attainability of $\rho(\g)$.

The proof is inspired by ideas employed in~\cite{sta-wac_26acc} for the linear case $p=2$, and it uses the following elementary lemmas which were proven therein; see \cite[Lemmas 7--9]{sta-wac_26acc}. In the course of the proof, we will make use of the \textit{middle operator} $\midd$ defined on $C(\Z)$ by
\begin{equation}
    \label{eq:def:mid}
    \midd u_n:=\frac{u_{n+1}+u_{n}}{2} \quad\text{ for all } n\in\Z.
\end{equation}

\begin{lemma} \label{lem:asymptotic_beh}
    Let $k\in\N$, $\lambda\in\R\setminus\N_0$, and $g$ be a sequence with the asymptotic expansion
    \begin{equation} \label{eq:lem:asymtotic_beh}
        g_{n}=\sum_{j=0}^{k}a_{j}n^{\lambda-j}+\bigO{n^{\lambda-k-1}}, \quad\text{ as } n\to\infty,
    \end{equation}
    for some $a_{j}\in\R$ with $a_{0}\neq0$. Then the following claims hold true.
    \begin{enumerate}[(i)]
    \item For all $m\in\Z$, there exist coefficients $a_{j}^{(m)}\in\R$ such that 
    \[
        \shift^{m} g_{n} = a_{0}n^{\lambda} + \sum_{j=1}^{k} a_{j}^{(m)} n^{\lambda-j} + \bigO{n^{\lambda-k-1}}, \quad\text{ as } n\to\infty.
    \] 
    \item For all $m\in\N_{0}$, there exist coefficients $b_{j}^{(m)}\in\R$ with $b_{0}^{(m)}\neq0$ such that 
    \[
        \midd^{m} g_{n}= \sum_{j=0}^{k} b_{j}^{(m)} n^{\lambda-j} + \bigO{n^{\lambda-k-1}}, \quad\text{ as } n\to\infty.
    \] 
    \item For all $m\in\{0,\dots,k\}$, there exist coefficients $c_{j}^{(m)}\in\R$ with $c_{m}^{(m)}\neq0$ such that 
    \[
        \ddiv^{m} g_{n}= \sum_{j=m}^{k} c_{j}^{(m)} n^{\lambda-j} + \bigO{n^{\lambda-k-1}}, \quad\text{ as } n\to\infty.
    \] 
    \end{enumerate}
\end{lemma}

\begin{lemma} \label{lem:MVT}
    Let $n\in\Z$, $m\in\N$ and $g \in C([n,n+m])\cap C^m((n,n+m))$. Then there exists $\xi \in (n,n+m)$ such that
    \[
        \ddiv^{m}g_{n}=g^{(m)}(\xi),
    \] 
    where we denoted $g_{n}:=g(n)$.
\end{lemma}

\begin{lemma} \label{lem:Leibniz}
    For all $m\in\N_{0}$ and sequences $f,g\in C(\Z)$, we have
    \[
        \ddiv^{m}(fg)=\sum_{j=0}^{m}\binom{m}{j}\left(\ddiv^{j}\midd^{m-j}f\right)\left(\ddiv^{m-j}\midd^{j}g\right),
    \]
    where the multiplication of sequences is to be understood pointwise.
\end{lemma}

The claim of Lemma \ref{lem:asymptotic_beh}(iii) asserts that the discrete derivative $\ddiv\g$, for $\g \in C(\Z)$ of the form \eqref{eq:lem:asymtotic_beh}, can be treated in a manner analogous to the continuous derivative $\g'$, once $\g$ is naturally extended to a function on $(0,\infty)$, as far as the asymptotic behavior is concerned. Lemma~\ref{lem:MVT} generalizes the classical \textit{Mean Value Theorem} and further elucidates the relationship between discrete and continuous derivatives. Lemma~\ref{lem:Leibniz} is the \textit{Leibniz formula} for the discrete divergence.

\begin{proof}[a) Proof of criticality of $\rho(\g)$:]
Suppose that $\tilde{\rho}$ is a discrete $p$-Hardy--Rellich--Birman weight satisfying $\tilde{\rho}_n\ge\rho_n(\g)$ for all $n\ge\ell$. Employing inequality \eqref{eq:thm:p-Hardy_l} for the weight $\rho(\g)$ and inequality \eqref{eq:p-hrb_ineq_gener} for the weight $\tilde{\rho}$, we find that
\begin{equation}
    \label{eq:criticality:1}
    0 
      \leq \sum_{n=\ell}^{\infty} \left(\tilde{\rho}_{n}-\rho_{n}(\g)\right)|u_{n}|^{p}
      \lesssim \sum_{k=1}^{\ell}R_{k}^{(\ell)}(p,\g,u)
\end{equation}
for all sequences $u\in\H_0^\ell$.

A quick observation of the remainders determined by formula \eqref{eq:thm:p-Hardy_l_rem} reveals that they are all simultaneously annihilated if we set $u=\g$, i.e. $R_{k}^{(\ell)}(p,\g,\g)\equiv0$ for all $k\in\{1,\dots,\ell\}$ and $p>1$. However, we cannot directly substitute $u=\g$ to \eqref{eq:criticality:1} since $\g\notin\H_0^\ell$. To this end, we introduce the following regularization of $\g$. First, let us define a smooth cut-off function $\xi^N$, for any $N\ge2$, as
\begin{equation} \label{eq:criticality:reg}
 \xi^{N}(x):=
    \begin{cases}
         1 & \text{ if } x\in(0,N],\\
         \eta\left(\frac{2\log N-\log x}{\log N}\right) & \text{ if } x\in(N, N^2],\\
         0 & \text{ if }  x\in(N^{2},\infty),\\
     \end{cases}
\end{equation}
where $\eta$ is a smooth function, such that $\eta\equiv0$ on $(-\infty,\varepsilon)$ and $\eta\equiv1$ on $(1-\varepsilon,\infty)$ for an arbitrary fixed $\varepsilon\in(0,1/2)$. Next, we define $u^N:=\xi^N\g$ for all $N\ge2$, where $\xi_n^N:=\xi^N(n)$ for all $n\ge1$. Then $u^{N}\in\H_{0}^{\ell}$ and $u^N\to\g$ pointwise as $N\to\infty$ since $\xi^N\to1$ pointwise as $N\to\infty$.

We will show that, with the chosen regularization, we have the estimates
\begin{equation} \label{eq:criticality:estimate}
    R_{k}^{(\ell)}(p;\g,u^{N}) \lesssim 
    \begin{cases}
    1/\log^{p-1}\! N &\text{ if }p\in(1,2],\\
    1/\log N &\text{ if }p\in(2,\infty),
    \end{cases}
\end{equation}
for all $k\in\{1,\dots,\ell\}$. Here $\lesssim$ means inequality $\le$ up to a multiplicative constant which may depend on $p$ and $\ell$, but is $N$-independent. Estimates from~\eqref{eq:criticality:estimate} together with Fatou's lemma and \eqref{eq:criticality:1} imply
\[
    \sum_{n=\ell}^{\infty}  \left(\tilde{\rho}_{n}-\rho_{n}(\g)\right) \g_n^p \leq 
    \liminf_{N\to\infty} \sum_{n=\ell}^{\infty}  \left(\tilde{\rho}_{n}-\rho_{n}(\g)\right) |u_n^N|^p = 0.
\]
Bearing in mind that all the terms in the sum on the left are non-negative and $\g_{n}>0$ for all $n\ge \ell$ by \eqref{assump:1}, we conclude that $\tilde{\rho}_{n}=\rho_{n}(\g)$ for all $n\ge\ell$, which proves the criticality of $\rho(\g)$.

In the rest of the proof of part~a), we verify estimates~\eqref{eq:criticality:estimate}. Replacing the gradients by divergences in \eqref{eq:thm:p-Hardy_l_rem}, we find that
\begin{equation}
    R_{k+1}^{(\ell)}(p;\g,u^N)
      = \sum_{n=\ell-k}^{\infty} \frac{-\lap_p^{(\ell-k-1)}\ddiv^{k+1}\g_{n}}{(\ddiv^{k+1}\g_{n})^{p-1}} \,\mathcal{R}(p;\ddiv^{k}\g,\ddiv^{k} u^N)_n
\label{eq:R_k+1_aux_inproof}
\end{equation}
for all $k\in\{0,\dots,\ell-1\}$. We proceed to analyze the asymptotic behavior of the summands as $n\to\infty$, with $N\ge 2$ fixed.

We begin by deducing an upper bound of the first factor. Specifically, we show that
\begin{equation}
    \label{eq:pr:criticality:1}
    \frac{-\lap_p^{(\ell-k-1)}\ddiv^{k+1}\g_{n}}{(\ddiv^{k+1}\g_{n})^{p-1}} \lesssim n^{p(k-\ell+1)}
\end{equation}
for any $n\ge\ell-k$ and $k\in\{0,\dots,\ell-1\}$. By assumption~\eqref{assump:4}, Lemma~\ref{lem:asymptotic_beh}(iii), and the generalized binomial formula, we get the expansion 
\begin{equation} \label{eq:div^p-1:asymp}
    (\ddiv^m\g_n)^{p-1}=n^{(p-1)(\ell-m-1/p)}\left[\,\sum_{j=0}^{2\ell-m}\frac{\gamma_j}{n^{j}} + \bigO{\frac{1}{n^{2\ell-m+1}}}\right], \quad\text{ as } n\to\infty,
\end{equation}
for any $m\in\{0,\dots,2\ell\}$, where $\gamma_j$ are $n$-independent constants and $\gamma_0\neq0$. Next, using the shift operator \eqref{eq:def:shift_mid} and definition \eqref{eq:def_p-lap} of the $p$-Laplacian, we may rewrite the numerator as
\[
    -\lap_p^{(\ell-k-1)}\ddiv^{k+1}\g_n = (-1)^{\ell+k+1} \,\shift^{k-\ell+1}\,\ddiv^{\ell-k-1}\,(\ddiv^\ell\g_n)^{p-1}.
\]
Hence, an application of claims~(i) and~(iii) of Lemma~\ref{lem:asymptotic_beh} to the sequence $(\ddiv^\ell\g)^{p-1}$ with asymptotic expansion~\eqref{eq:div^p-1:asymp} yields
\[
    -\lap_p^{(\ell-k-1)}\ddiv^{k+1}\g_n \lesssim n^{k-\ell+1/p}.
\]
This inequality and~\eqref{eq:div^p-1:asymp} with $m=k+1$ imply~\eqref{eq:pr:criticality:1}.

Further, we examine the asymptotic behavior of the second factor. Namely, we prove that
\begin{equation} \label{eq:pr:criticality:2}
     \mathcal{R}(p;\ddiv^k\g,\ddiv^k u^N)_n \lesssim \frac{n^{p(\ell-k-1)-1}}{\log^{\min(p,2)}N}
\end{equation}
for all $p>1$, where again the unspecified constant is independent of $n$ and $N$. First, we establish the inequality
\begin{equation} \label{eq:pr:criticality:sw_estimate}
    \left|\sqrt{\frac{\ddiv^{k} \g_{n}}{\ddiv^k \g_{n+1}}}\ddiv^k u_{n+1}^N - \sqrt{\frac{\ddiv^k \g_{n+1}}{\ddiv^k \g_{n}}} \ddiv^k u_n^N\right| \lesssim\frac{n^{\ell-k-1-1/p}}{\log{N}}
\end{equation}
for all $N\ge2$, $n\ge\ell$, and $k\in\{0,\dots,\ell-1\}$. In fact, this result implicitly appeared already in~\cite[Sec.~2.4]{sta-wac_26acc} but for the sake of completeness, we repeat the derivation. By assumption~\eqref{assump:4} and claims~(i) and~(iii) of Lemma~\ref{lem:asymptotic_beh}, we find that
\[
    \sqrt{\frac{\ddiv^{k}\g_{n}}{\ddiv^{k}\g_{n+1}}}\lesssim1,
\]
and therefore the left-hand side of~\eqref{eq:pr:criticality:sw_estimate} can be estimated from above by 
\[
    \left|\ddiv^{k}u_{n+1}^N - \frac{\ddiv^{k}\g_{n+1}}{\ddiv^{k}\g_{n}}\ddiv^{k}u_{n}^N\right|.
\]
Recalling that $u^{N}=\xi^{N}\g$ and using the Leibniz formula from Lemma~\ref{lem:Leibniz} the last expression can be further estimated by
\begin{align*}
      \sum_{j=0}^{k}\binom{k}{j}\ddiv^{j}\midd^{k-j}\g_{n+1} \left|\ddiv^{k-j}\midd^{j}\xi^{N}_{n+1} - \frac{\ddiv^{k}\g_{n+1}}{\ddiv^{k}\g_{n}} \frac{\ddiv^{j}\midd^{k-j}\g_{n}}{\ddiv^{j}\midd^{k-j}\g_{n+1}}\,\ddiv^{k-j}\midd^{j}\xi^{N}_{n}\right|.
\end{align*}
Yet another application of~\eqref{assump:4} and formulas from Lemma~\ref{lem:asymptotic_beh} gives
\[
    \ddiv^{j}\midd^{k-j}\g_{n+1}\lesssim n^{\ell-j-1/p} 
    \quad\text{ and }\quad \frac{\ddiv^{k}\g_{n+1}}{\ddiv^{k}\g_{n}} \frac{\ddiv^{j}\midd^{k-j}\g_{n}}{\ddiv^{j}\midd^{k-j}\g_{n+1}} = 1 + p_n^{(k,j)},
\]
where $p_{n}^{(k,j)}\lesssim 1/n$, with $p^{(k,k)}\equiv0$. Altogether, we deduce the upper bound for the left-hand side of~\eqref{eq:pr:criticality:sw_estimate} in the form
\begin{equation} \label{eq:criticality:2}
    \sum_{j=0}^{k} \binom{k}{j} n^{\ell-j-1/p} \biggl(\left|\ddiv^{k-j+1}\midd^{j}\xi^{N}_{n+1}\right| + \left|p_{n}^{(k,j)}\,\ddiv^{k-j}\midd^{j}\xi^{N}_{n}\right|\biggr)
\end{equation}
for all $N\ge2$, $n\ge\ell$, and $k\in\{0,\dots,\ell-1\}$.

In order to arrive at the desired estimate~\eqref{eq:pr:criticality:sw_estimate}, we have to inspect asymptotic behavior of all the $\ddiv$ terms from~\eqref{eq:criticality:2}. To this end, observe that for all $x\in[N,N^2]$, we have
\[
    (\xi^N)'(x)=-\eta'\left(\frac{2\log{N}-\log{x}}{\log{N}}\right) \frac{1}{x\log{N}}.
\]
Thus, $|(\xi^N)'(x)|\lesssim 1/(x\log{N})$, as $\eta'$ is majorized by $\max_{x\in[0,1]}|\eta'(x)|$. It is straightforward to generalize this bound to higher-order derivatives; we find
\[
    \bigl|(\xi^{N})^{(m)}(x)\bigr|\lesssim\frac{1}{x^{m}\log N}
\]
for all $m\in\N$ and $x>0$. Consequently, by Lemma~\ref{lem:MVT}, we obtain 
\[
    \bigl|\ddiv^m\xi^N_{n}\bigr|\lesssim\frac{1}{n^m\log{N}},
\]
for any $m,n\in\N$ and $N\ge2$. Since the right-hand side is a decreasing function of $n$, we also have 
\begin{equation} \label{eq:criticality:xi_bound}
    \bigl|\ddiv^m\midd^j\xi_n^N\bigr|\lesssim\frac{1}{n^m\log{N}}
\end{equation}
for any $j\in\N_0$. From these estimates, we infer further upper bound for~\eqref{eq:criticality:2} in the form
\[
    \sum_{j=0}^{k-1} \binom{k}{j} n^{\ell-j-1/p}\left(\frac{1}{n^{k-j+1}\log{N}} + \frac{1}{n}\frac{1}{n^{k-j}\log{N}}\right)+n^{\ell-k-1/p} \frac{1}{n\log{N}} \lesssim \frac{n^{\ell-k-1-1/p}}{\log{N}},
\]
arriving at \eqref{eq:pr:criticality:sw_estimate}.

If $p\in(1,2]$, inequality~\eqref{eq:pr:criticality:2} follows immediately from~\eqref{eq:pr:criticality:sw_estimate} and definition~\eqref{eq:thm:p-Hardy_1_rem} of the corresponding remainders. If $p\in(2,\infty)$, we must also estimate the second factor in the defining formula~\eqref{eq:thm:p-Hardy_1_rem}. In particular, for any $k\in\{0,\dots,\ell-1\}$, we show that
\begin{equation} \label{eq:pr:criticality:3}
    \bigl|\ddiv^{k+1}u_{n}^N\bigr|+\frac{\bigl|\ddiv^k u_{n}^N\bigr|}{\ddiv^k\g_{n}}\ddiv^{k+1}\g_{n} \lesssim n^{\ell-k-1-1/p}
    \quad\text{ for all } n\ge\ell-k.
\end{equation}
By Lemma~\ref{lem:Leibniz}, 
\[
    \ddiv^k u_n^N = \ddiv^k(\xi^N\g)_n=\sum_{j=0}^{k}\binom{k}{j} (\ddiv^{j}\midd^{k-j}\g_{n}) (\ddiv^{k-j}\midd^{j}\xi^{N}_{n})
\]
for any $k\geq0$. Further, using claims~(ii) and~(iii) of Lemma~\ref{lem:asymptotic_beh} and inequality~\eqref{eq:criticality:xi_bound}, we infer that
\[
    \ddiv^{j}\midd^{k-j}\g_{n} \lesssim n^{\ell-j-1/p} \quad\text{ and }\quad \ddiv^{k-j}\midd^{j}\xi^{N}_{n} \lesssim n^{j-k}
\]
for all $N\ge2$, $n\ge\ell-k$, and $j\in\{0,\dots,k\}$. Note that the second estimate for $j=k$ follows from the obvious inequality $\xi^N\le1$. Altogether, we get
\[
    \ddiv^k u_n^N \lesssim n^{\ell-k-1/p}.
\]
Furthermore, using Lemma~\ref{lem:asymptotic_beh}(iii) once again, we find that
\[
    \frac{\ddiv^{k+1}\g_{n}}{\ddiv^k\g_n}\lesssim \frac{1}{n}.
\]
Combining the above inequalities yields the estimate~\eqref{eq:pr:criticality:3}. Recalling definition~\eqref{eq:thm:p-Hardy_1_rem}, inequalities~\eqref{eq:pr:criticality:sw_estimate} and~\eqref{eq:pr:criticality:3} amount to the desired inequality~\eqref{eq:pr:criticality:2} for $p>2$.

Finally, the inequalities~\eqref{eq:pr:criticality:1} and~\eqref{eq:pr:criticality:2} applied in \eqref{eq:R_k+1_aux_inproof}, together with the observation that $\mathcal{R}(p;\ddiv^k\g,\ddiv^k u^N)_n=0$ for all $n<N-k$ and $n>N^{2}$,
yield
\begin{align*}
    R_{k+1}^{(\ell)}(p;\g,u^N)
      &=\sum_{n=N-k}^{N^2}\!\frac{-\lap_p^{(\ell-1-k)}\ddiv^{k+1}\g_{n}}{(\ddiv^{k+1}\g_{n})^{p-1}}\,\mathcal{R}(p;\ddiv^k\g;\ddiv^k u^N)_n \lesssim \frac{1}{\log^{\min(p,2)}{N}} \sum_{n=N-k}^{N^2} \frac{1}{n}\\
      &\lesssim \frac{1}{\log^{\min(p,2)}{N}}\int_N^{N^2}\frac{\dd n}{n} = \frac{1}{\log^{\min(p,2)-1}{N}}
\end{align*}
for all $N\geq\ell$ and $k\in\{0,\dots,\ell-1\}$. The last estimate coincides with~\eqref{eq:criticality:estimate}, and the proof of criticality is complete.
\end{proof}

\begin{proof}[b) Proof of non-attainability of $\rho(\g)$:]
Under the additional condition of strict positivity in assumption~\eqref{assumpp:3}, we first extend the lower bound in~\eqref{eq:thm:p-Hardy_l} from $\H_0^\ell$ to sequences in $H^\ell$ for which equality is attained in~\eqref{eq:p-hrb_ineq_gener} and both sides are finite. For this extension, we need the following auxiliary result.

\begin{lemma}
    \label{lem:density_of_grad}
    The range of the operator $\grad^\ell|_{\H_0^\ell}$ is dense in $\H^{\ell,p}$.
\end{lemma}

\begin{proof}[Proof of Lemma~\ref{lem:density_of_grad}]
    Clearly, $\grad^\ell(\H_0^\ell)\subset\H_0^\ell\subset\H^{\ell,p}$ by definition 
    \eqref{eq:def_grad_div_Z} of the discrete gradient. We prove that if $v\in\H^{\ell,q}$ satisfies
    \begin{equation}
        \label{eq:lem:density:inn_prod}
        \bigl\langle{v},{\grad^\ell u}\bigr\rangle=\sum_{n=\ell}^{\infty}\overline{v_n}\,\grad^\ell u_n =0
        \quad\text{ for all }u\in\H_0^\ell,
    \end{equation}
    then $v\equiv0$. The density of $\grad^\ell(\H_0^\ell)$ in $\H^{\ell,p}$ then follows from the Hahn--Banach Theorem. 
    
    Since the algebraic adjoint to $\grad$ is $-\ddiv$, we may rewrite the equality from~\eqref{eq:lem:density:inn_prod} as
    \[
    \bigl\langle (-1)^{\ell}\ddiv^{\ell}v,u\bigr\rangle=0.
    \]
    Taking $u=\delta_n$ for all $n\geq\ell$, where $\delta_n$ is the $n$-th vector of the standard basis of $\H_0^\ell$, in the last equation, we observe that $v$ solves the linear difference equation of order $\ell$ with constant coefficients:
    \[
        (-1)^\ell\ddiv^\ell v_n=\sum_{j=0}^\ell \binom{\ell}{j}(-1)^jv_{n+j}=0
        \quad\text{ for all } n\ge\ell.
    \]
    Since the fundamental system of this linear difference equation is given by the sequences $1, n, \dots, n^{\ell-1}$, we deduce that
    \[
        v_n = \sum_{j=0}^{\ell-1} c_j n^j
        \quad\text{ for all } n \ge \ell
    \]
    and some $c_j \in \C$. Taking also into account that $v \in \ell^q(\Z)$, we conclude that $c_j = 0$ for all ${j \in \{0,\dots,\ell-1\}}$, i.e. $v \equiv 0$. The proof of Lemma~\ref{lem:density_of_grad} is complete.
\end{proof}

Now suppose that $u\in H^\ell$ fulfills inequality~\eqref{eq:p-hrb_ineq_gener} as equality whose (both) sides are finite, i.e.
\begin{equation}
    \label{eq:non-att_equality}
    \sum_{n=\ell}^{\infty} \bigl|\grad^\ell u_n\bigr|^p=\sum_{n=\ell}^{\infty} \rho_n(\g)|u_n|^p < \infty.
\end{equation}
Consequently, $\grad^\ell u\in\H^{\ell,p}$ and therefore, by Lemma~\ref{lem:density_of_grad}, there exists a sequence of $u^N\in\H_0^\ell$ such that $\grad^\ell u^N\to\grad^\ell u$ as $N\to\infty$ in the norm of $\ell^p(\Z)$. In particular, the sequence is Cauchy in $\ell^p(\Z)$.

Due to Theorem~\ref{thm:p-Hardy_weight}, the inequality~\eqref{eq:p-hrb_ineq_gener} holds for $u=u^{N}-u^{M}\in\H_0^\ell$ and the non-negative weight $\rho(\g)$, which means that
\[
    \sum_{n=\ell}^{\infty}\rho(\g)\left|u_n^N-u_n^M\right|^{p}\leq\sum_{n=\ell}^{\infty}\left|\grad^{\ell}u_n^N-\grad^{\ell}u_n^M\right|^{p} \quad\text{ for all } M,N\in\N
\]
Consequently, the sequence $\{\sqrt[\uproot{2}\leftroot{-2}p]{\rho(\g)}\,u^N\}_{N=1}^\infty$ is Cauchy, too, and so convergent in $\ell^p(\Z)$. Denote by $v$ the $\ell^p$--limit of $\sqrt[\uproot{2}\leftroot{-2}p]{\rho(\g)}\,u^N$ as $N\to\infty$. It is immediate that $v\in\H^{\ell,p}$ and that, for every $n\in\Z$, we have $\sqrt[\uproot{2}\leftroot{-2}p]{\rho_n(\g)}\,u_n^N \to v_n$ as $N\to\infty$. Since $\rho_n(\g)>0$ for all $n\geq \ell$ by the strict inequalities in assumption~\eqref{assumpp:3}, we may put $\widetilde{u}_n:=v_n/\sqrt[\uproot{2}\leftroot{-2}p]{\rho_n(\g)}$ for $n\geq\ell$ and $\widetilde{u}_{n}:=0$ for $n<\ell$. Then $\widetilde{u}\in H^\ell$ and $u_n^N \to \widetilde{u}_n$ as $N\to\infty$ for all $n\in\Z$.

Moreover, the limits of both $\ell^p$-convergent sequences must coincide with their pointwise limits, i.e.
\[
    \grad^\ell u^N \to \grad^\ell \widetilde{u} 
    \quad\text{ and }\quad
    \sqrt[\uproot{2}\leftroot{-2}p]{\rho(\g)}\,u^N\to\sqrt[\uproot{2}\leftroot{-2}p]{\rho(\g)}\,\widetilde{u}
\]
in $\ell^p(\Z)$ as $N\to\infty$. In particular, we have $\grad^\ell u=\grad^\ell\widetilde{u}$, which implies $u=\widetilde{u}$ since the linear operator $\grad^\ell$ is injective on $H^\ell$. Indeed, if $w$ solves the difference equation
\[
    \grad^\ell w_n=0
    \quad\text{ for all }n\ge\ell,
\]
with the boundary condition $w_n=0$ for all $n<\ell$, corresponding to the assumption $w\in H^\ell$, we recursively obtain $w_n=0$ for all $n\in\Z$.

The lower estimate from~\eqref{eq:thm:p-Hardy_l} with $u=u^{N}$ yields
\[
    0 
      \le \sum_{k=1}^{\ell}R_{k}^{(\ell)}(q;\g,u^N)
      \lesssim \sum_{n=\ell}^{\infty} \bigl|\grad^\ell u_n^N\bigr|^p - \sum_{n=\ell}^{\infty}\rho_n(\g)\bigl|u_n^N\bigr|^p
\]
for all $N\in\N$. Sending $N\to\infty$ and using the assumption~\eqref{eq:non-att_equality}, we find that all the non-negative remainders on the left-hand side must vanish as $N\to\infty$. 
In particular, for $k=\ell$, it follows that
\[
R_{\ell}^{(\ell)}(q;\g,u^N)\to0 \quad\text{ as } N\to\infty.
\]
Recalling definitions~\eqref{eq:thm:p-Hardy_l_rem} and \eqref{eq:thm:p-Hardy_1_rem}, an application of Fatou's Lemma yields
\[
    \mathcal{R}(q;\grad^{\ell-1}\g,\grad^{\ell-1}u)_n =0
\]
for all $n\ge\ell$. This is the same equation as in~\eqref{eq:recur_weak-non-att_inproof} and, by using exactly the same argument as in the proof of Theorem~\ref{thm:p-Hardy_weight}, we infer that $u_n=c\,\g_n$ for a constant $c\in\C$ and all $n\ge \ell$.

Next, we inspect the asymptotic behavior of $\rho_{n}(\g)=-\lap_p^{(\ell)}\g_n/\g_n^{p-1}$ for $n\to\infty$. Recalling definition~\eqref{eq:def_p-lap} of the $p$-Laplacian, the numerator can be equivalently written as
\[
    -\lap_p^{(\ell)}\g_n = (-1)^\ell \shift^{-\ell}\ddiv^\ell (\ddiv^\ell\g_n)^{p-1},
\]
where $\shift$ is the shift operator~\eqref{eq:def:shift_mid}. By claims~(i) and~(iii) of Lemma~\ref{lem:asymptotic_beh}, together with expansion~\eqref{eq:div^p-1:asymp} with $m=\ell$, we deduce that
\begin{equation}
    \label{eq:pr:optimal_asympt_rho}
    \rho_n(\g) \gtrsim \frac{n^{1/p-\ell-1}}{n^{(p-1)(\ell-1/p)}} = \frac{1}{n^{\ell p}}.
\end{equation}

Taking also assumption~\eqref{assump:4} into account, we find 
\[
    \sum_{n=\ell}^{\infty} \rho_n(\g) \g_n^p \gtrsim \sum_{n=\ell}^{\infty} \frac{1}{n} =\infty.
\]
Recalling that $u_n=c\,\g_n$ for all $n\ge \ell$, the last equation together with assumption~\eqref{eq:non-att_equality} imply that $c=0$, i.e. $u\equiv0$. The proof of non-attainability is complete.
\end{proof}

\begin{proof}[c) Proof of optimality near infinity of $\rho(\g)$:]
Fix an arbitrary $m\ge\ell$. We prove that $\rho(\g)$ enjoys property \eqref{eq:opt_near_inf}
by constructing a sequence $\{u^N\}_{N\ge m}\subset\H_0^m$ such that
\begin{equation} \label{eq:pr:opt_near_inf}
    \lim_{N\to\infty} \frac{\sum_{k=1}^{\ell} R_k^{(\ell)}(p;\g,u^N)}{\sum_{n=\ell}^{\infty}\rho_n(\g)\big|u_n^N\big|^p} = 0.
\end{equation}
which follows from the upper bound of~\eqref{eq:thm:p-Hardy_l}.
Similarly as in the proof of criticality, the idea is to find a suitable regularization of the parameter sequence $\g$ so that it becomes an element of the space $\H_0^m$. Namely, we set $u^N:=\xi^N\g$, where $\xi$ is a smooth bump function defined by
\[
    \xi^{N}(x):=
    \begin{cases}
        0 & \text{ if }  x\in(0,N],\\
        \eta\left(\frac{\log x-\log N}{\log N}\right) & \text{ if }  x\in(N,N^{2}],\\
        1 & \text{ if }  x\in(N^{2},2N^{2}],\\
        \eta\left(\frac{\log(2N^{3})-\log x}{\log N}\right) & \text{ if }  x\in(2N^{2},2N^{3}],\\
        0 & \text{ if } x\in(2N^{3},\infty),
    \end{cases}
\]
where the function $\eta$ exhibits the same properties as in~\eqref{eq:criticality:reg}. Then  $u^N\in\H_0^m$ for all $N\ge m$.

Analogously as in~\eqref{eq:criticality:xi_bound}, we may show that
\[
    \left|\midd^j\xi_n^N\right|\lesssim1 \quad\text{ and }\quad \left|\ddiv^k\midd^j\xi_n^N\right|\lesssim\frac{1}{n^k\log{N}}
\]
for all $k,n\in\N$, $j\in\N_0$, and $N\ge m$, where $\midd$ is the middle operator \eqref{eq:def:mid}. Therefore, the same arguments as in the proof of criticality apply, and we obtain the estimates~\eqref{eq:criticality:estimate} even with $u^N$ as defined above. Consequently, it follows the upper bound
\[
    R_{k+1}^{(\ell)}(p;\g,u^N)
      \lesssim \frac{1}{\log^{\min(p,2)}{N}} \left(\sum_{n=N-k}^{N^2} \frac{1}{n} +\sum_{n=2N^2-k}^{2N^3} \frac{1}{n}\right) \lesssim \frac{1}{\log^{\min(p,2)-1}{N}}
\]
for all $k\in\{0,\dots,\ell-1\}$ and $N$ sufficiently large, where the unspecified constant does not depend on $N$, but may depend on $p$ and $\ell$. Hence the numerator in~\eqref{eq:pr:opt_near_inf} tends to zero as $N\to\infty$.

In order to finish the proof, we show that the denominator in~\eqref{eq:pr:opt_near_inf} is bounded from below by a positive constant. By employing estimate \eqref{eq:pr:optimal_asympt_rho} and assumption \eqref{assump:4}, we find
\[
    \sum_{n=\ell}^{\infty}\rho_n(\g)\big|u^N_n\big|^p\geq\sum_{n=N^2}^{2N^2}\rho_n(\g)\g_n^p\gtrsim\sum_{n=N^2}^{2N^2}\frac{n^{\ell p-1}}{n^{\ell p}}\ge\int_{N^2}^{2N^2} \frac{\dd n}{n} = \log{2},
\]
as claimed. The proof of optimality near infinity and therefore of Theorem~\ref{thm:optimal_p-Hardy} is complete.
\end{proof}

\section{Proofs of the concrete inequalities}\label{sec:proofs_concrete}

This section proves Theorem~\ref{thm:optimal_rho_LB_form} in two steps. First, we establish that the weight $\rho^{(\ell,p)}$ defined in~\eqref{eq:optimal_rho_LB_form} is an optimal discrete $p$-Hardy--Rellich--Birman weight by showing that the sequence $\g^{(\ell,p)}$ from~\eqref{eq:def:g} satisfies the four assumptions of Theorem~\ref{thm:optimal_p-Hardy}; 
the second equality in~\eqref{eq:optimal_rho_LB_form} is proven simultaneously. Second, we show that the weight $\rho^{(\ell,p)}$ improves on the classical $p$-Birman weight, i.e., we prove inequality~\eqref{eq:ineq_improve_birman}.

\subsection{Proof of the optimality of \texorpdfstring{$\rho^{(\ell,p)}$}{r}}

First, we show that $\g^{(\ell,p)}$ behaves nicely under the action of the discrete divergence, i.e. it fulfills a full analogy to the classical differentiation formula for monomials
\[
    \quad \frac{\dd^k \mathfrak{m}^{(\ell,p)}}{\dd x^k}=\left(\ell-k+1/q\right)_{k}\, \mathfrak{m}^{(\ell-k,p)}, \quad\text{ with } \mathfrak{m}^{(\ell,p)}(x):=x^{\ell-1/p},
\]
where $(a)_\ell :=a(a+1)\dots(a+\ell-1)$ is the Pochhammer symbol.

\begin{lemma}
    \label{lem:div_g}
    For all $k\in\N_0$ and $n\in\Z$, we have the identity
    \begin{equation}
        \label{eq:lem:div_g}
        \ddiv^k\g^{(\ell,p)}_n=(\ell-k+1/q)_k\,\g^{(\ell-k,p)}_n.
    \end{equation}
\end{lemma}
\begin{proof}
It follows from definition~\eqref{eq:def_grad_div_Z} of the discrete divergence that
\begin{equation}\label{eq:div^k_recur}
    (-1)^{k}\ddiv^k u_n = \sum_{j=0}^k (-1)^{j}\binom{k}{j} u_{n+j} \quad \text{ for all } n\in\Z.
\end{equation}
An application of~\eqref{eq:div^k_recur} to $u=\g^{(\ell,p)}$, defined by \eqref{eq:def:g}, and simple manipulations imply
\[
    (-1)^{k}\ddiv^k \g^{(\ell,p)}_n
      = \sum_{j=0}^k (-1)^j\binom{k}{j} \frac{\Gamma(n + j + 1/q)}{\Gamma(n - \ell + j + 1)}
      = \Gamma(n+1/q)\sum_{j=0}^k \binom{k}{j} \frac{(-1)^j(n + 1/q)_{j}}{\Gamma(n - \ell + j + 1)}
\]
for all $n \in \Z$. The last sum can be simplified with the aid of the Chu--Vandermonde identity~\cite[Eq.~(15.4.24)]{dlmf}, which reads
\begin{equation}
 \sum_{j=0}^{k}\binom{k}{j}\frac{(-1)^j(b)_{j}}{\Gamma(c+j)}=\frac{(c-b)_{k}}{\Gamma(c+k)} \quad \text{ for all } b,c\in\C,
\label{eq:chu-van_id}
\end{equation}
and can be easily verified by induction in $k\in\N_0$. Applying~\eqref{eq:chu-van_id}, we get
\[
    \ddiv^k \g^{(\ell,p)}_n 
      = (-1)^k (-\ell-1/q+1)_{k}\,\frac{\Gamma(n+1/q)}{\Gamma(n-\ell+k+1)}
      = (\ell-k+1/q)_{k}\,\g^{(\ell-k,p)}_n,
\]
and the proof of Lemma~\ref{lem:div_g} is complete.
\end{proof}

The explicit formula~\eqref{eq:lem:div_g} plays a crucial role, as it yields a closed-form expression
\begin{equation}
    \label{eq:div^l_g_formula}
    \ddiv^\ell\g^{(\ell,p)}_n=\left(\frac{1}{q}\right)_{\!\ell}\,\frac{\Gamma(n+1/q)}{\Gamma(n+1)}
    \quad\text{ for all } n \ge 0,
\end{equation}
which is a key identity in showing that the chosen parameter sequence $\g^{(\ell,p)}$ satisfies all assumptions of Theorem~\ref{thm:optimal_p-Hardy}.

As the next step, we verify that the ratio of Gamma functions, which appears in~\eqref{eq:div^l_g_formula}, defines a logarithmically completely monotone function on $\R_{+}$. Recall that a smooth function $g:\R_{+}\to\R_{+}$ is called 
\emph{logarithmically completely monotone} if $h=-(\log g)'$ is \emph{completely monotone} on $\R_+$, which, in turn, means that 
\begin{equation}
(-1)^k h^{(k)}(x)\ge0 \quad\text{ for all } k\in\N_0 \text{ and } x>0.
\label{eq:def_CM}
\end{equation}
For a comprehensive theory on completely monotone functions and related classes, we refer the reader to the monographs \cite{sch-son-von_12, wid_41}.
Here, we only need to know that, if $g$ is logarithmically completely monotone, then $g^{\alpha}$ is completely monotone on $\R_{+}$ for all $\alpha>0$ (this is not the case if $g$ is only completely monotone, see \cite{van-hae_96, van-hae_97}); this property is proven in \cite[Thm.~5.9]{sch-son-von_12}.

The following fact is well known and holds even in a greater generality, see \cite[Thm.~10]{alz_97}. We provide its short proof for the reader's convenience.

\begin{lemma}
    \label{lem:div_g_LCM}
    For any $q>1$, the function
    \begin{equation}
        \label{eq:def:G_fun}
        g(x):=\frac{\Gamma(x+1/q)}{\Gamma(x+1)}
    \end{equation}
    is logarithmically completely monotone on $\R_{+}$.
\end{lemma}
\begin{proof}
    Clearly, $g>0$ on $\R_{+}$ and so we need to verify that the function
    \[
        h(x):=-(\log g(x))' = \psi(x+1) - \psi(x+1/q)
    \]
    where $\psi:=\Gamma'/\Gamma$ is the digamma function, is completely monotone on $\R_{+}$.
    Using the integral representation~\cite[Eq.~(5.9.16)]{dlmf}
    \begin{equation}
        \label{eq:psi_int_rep}
        \psi(x) = -\gamma + \int_{0}^{\infty} \frac{\ee^{-t} - \ee^{-xt}}{1 - \ee^{-t}} \dd t
    \end{equation}
    for $x>0$, where $\gamma$ denotes the \emph{Euler--Mascheroni constant}, we obtain 
    \[
        h(x) = \int_{0}^{\infty} \ee^{-xt} \,\frac{\ee^{-t/q} - \ee^{-t}}{1 - \ee^{-t}} \dd t
    \]
    for all $x>0$. Consequently, for all $k\in\N_{0}$ and $x>0$, we have
    \[
        (-1)^k h^{(k)}(x)
        = \int_0^\infty t^k\ee^{-xt} \,\frac{\ee^{-t/q} - \ee^{-t}}{1 - \ee^{-t}} \dd t > 0
    \]
    since $(\ee^{-t/q} - \ee^{-t})/(1 - \ee^{-t})>0$ for all $t>0$, which follows from the assumption $q>1$.
\end{proof}

Now, we can proceed to the verification of the assumptions of Theorem~\ref{thm:optimal_p-Hardy} with the parameter sequence $\g^{(\ell,p)}$.

\begin{proof}[Proof of optimality of $\rho^{(\ell,p)}$:]
By definitions~\eqref{eq:def_grad_div_Z}, \eqref{eq:def:g}, and Lemma~\ref{lem:div_g}, we find that
  \[
        \grad^k \g^{(\ell,p)}_n = (\ell - k + 1/q)_{k}\,\g^{(\ell-k,p)}_{n-k} > 0
  \]
for all $n \ge \ell \ge k \ge 0$. Therefore $\g^{(\ell,p)}$ satisfies the assumption \eqref{assump:1}. Moreover, by definition~\eqref{eq:def_p-lap}, we have
\[
    \rho^{(\ell,p)}_{n}=\frac{(-1)^{\ell}\ddiv^{\ell}\left(\ddiv^{\ell}\g^{(\ell,p)}\right)_{n-\ell}^{p-1}}{\big(\g^{(\ell,p)}_{n}\big)^{p-1}} \quad\text{ for all } n\geq\ell,
\]
from which, when applying formulas \eqref{eq:div^l_g_formula}, \eqref{eq:div^k_recur}, and \eqref{eq:def:g}, we verify the second equality from \eqref{eq:optimal_rho_LB_form}.

Next, we verify that $\g^{(\ell,p)}$ meets assumptions \eqref{assump:2} and \eqref{assumpp:3}. As mentioned above, since the function $g$ defined by~\eqref{eq:def:G_fun} is logarithmically completely monotone, its arbitrary positive power is completely monotone on $\R_{+}$. It follows that 
\[
(-1)^{k}\frac{\dd^{k}g^{p-1}}{\dd x^{k}}(x)\geq0 \quad\text{ for all } k\in\N_0 \text{ and } x>0.
\]
These inequalities are actually strict, since otherwise $g$ would have to be a constant function on $\R_{+}$, which is not the case. This is a consequence of Bernstein's theorem, see \cite[Rem.~1.5]{sch-son-von_12}. Bearing in mind formulas \eqref{eq:def:G_fun} and \eqref{eq:div^l_g_formula}, an application of Lemma~\ref{lem:MVT} implies that 
\begin{equation}
            (-1)^k\ddiv^k (\ddiv^\ell \g^{(\ell,p)}_n)^{p-1}>0
            \quad\text{ for all } k, n\in\N_0.
\label{eq:p-lap_g_pos}
\end{equation}
It follows that $\g^{(\ell,p)}$ satisfies assumptions \eqref{assump:2} and \eqref{assumpp:3} with strict inequalities.

Lastly, recall that the ratio of Gamma functions admits the complete asymptotic expansion~\cite[Eq.~(5.11.13)]{dlmf}
\begin{equation}
    \label{eq:Gamma_ratio_exp}
    \frac{\Gamma(x+a)}{\Gamma(x+b)}
    \sim
    x^{a-b}
    \sum_{k=0}^\infty \frac{G_k(a,b)}{x^k},
    \quad\text{ as } x\to\infty,
    \text{ for all } a,b\in\R,
\end{equation}
where the coefficients $G_{n}(a,b)$ can be expressed in terms of the generalized Bernoulli polynomials; see~\cite[Eq.~(5.11.17)]{dlmf} for details. In particular, $G_0(a,b)=1$. Applying~\eqref{eq:Gamma_ratio_exp} to definition~\eqref{eq:def:g} of $\g^{(\ell,p)}$, one readily verifies that $\g^{(\ell,p)}$ also satisfies assumption~\eqref{assump:4}.

As $\g^{(\ell,p)}$ fulfills all the assumptions, we may apply Theorem~\ref{thm:optimal_p-Hardy}, which implies that $\rho^{(\ell,p)}$ is an optimal discrete $p$-Hardy--Rellich--Birman weight.
\end{proof}

\subsection{Proof of inequality~\texorpdfstring{\eqref{eq:ineq_improve_birman}}{ref}}

To complete the proof of Theorem~\ref{thm:optimal_rho_LB_form}, it remains to show that $\rho^{(\ell,p)}$ improves upon the classical $p$-Birman weight. Recall that, by~\eqref{eq:optimal_rho_LB_form}, we have the formula 
\[
\rho_n^{(\ell,p)}
          = \left(\frac{1}{q}\right)_{\!\ell}^{\!p-1} \left[\frac{\Gamma(n-\ell+1)}{\Gamma(n+1/q)}\right]^{p-1}\sum_{j=0}^{\ell}(-1)^{j}\binom{\ell}{j}\left[\frac{\Gamma(n-\ell+j+1/q)}{\Gamma(n-\ell+j+1)}\right]^{p-1} \quad\text{ for } n\geq\ell.
\]
In the following two lemmas, we separately estimate the prefactor and the sum itself. The latter is non-trivial because the sum involves terms of alternating sign.

\begin{lemma}
        \label{lem:LB_1}
        For all $n\ge\ell$, we have the inequality
        \[
            \frac{\Gamma(n-\ell+1)}{\Gamma(n+1/q)}>\frac{1}{n^{\ell-1/p}}.
        \]
\end{lemma}

\begin{proof}
First, we estimate
        \[
            \frac{\Gamma(n-\ell+1)}{\Gamma(n+1/q)}
              = \frac{1}{n(n-1)\cdots(n-\ell+1)}\,\frac{\Gamma(n+1)}{\Gamma(n+1/q)}
              > \frac{1}{n^{\ell}}\,\frac{\Gamma(n+1)}{\Gamma(n+1/q)}.
        \]
On the right-hand side, we apply Gautschi's inequality \cite[Eq.~(5.6.4)]{dlmf}, which states that
        \begin{equation}
            \label{eq:gautschi}
            x^{1-s} < \frac{\Gamma(x+1)}{\Gamma(x+s)}<(x+1)^{1-s}
            \quad\text{ for all } x>0 \text{ and }s\in(0,1), 
        \end{equation}
and the claim follows.
\end{proof}

Notice that the next lemma excludes the case $\ell=1$.

\begin{lemma}
        \label{lem:LB_2}
        Let $\ell\geq2$. Then the inequality
        \[
            \sum_{j=0}^{\ell}(-1)^{j}\binom{\ell}{j}\left[\frac{\Gamma(n-\ell+j+1/q)}{\Gamma(n-\ell+j+1)}\right]^{p-1}
              \geq \left(\frac{1}{q}\right)_{\!\ell}\,\frac{1}{n^{\ell+1/q}}
        \]
        holds for all $n\ge\ell$.
\end{lemma}
\begin{proof}
        The proof is divided into two steps. First, we replace the ratio of Gamma functions on the left-hand side of the inequality with a more tractable power sequence. Second, we obtain a lower bound for this new sequence that coincides with the right-hand side of the inequality.

        \emph{Step 1}: 
        We show that the inequality
        \begin{equation}
            \label{eq:lem:LB_2_pr:1}
            \sum_{j=0}^{\ell}(-1)^{j}\binom{\ell}{j}\left[\frac{\Gamma(n-\ell+j+1/q)}{\Gamma(n-\ell+j+1)}\right]^{p-1} 
              \ge \sum_{j=0}^{\ell}(-1)^{j}\binom{\ell}{j} \frac{1}{(n-\ell+j+1)^{1/q}}
        \end{equation}
        holds for all $n\ge\ell\ge0$. By~\eqref{eq:div^k_recur}, inequality \eqref{eq:lem:LB_2_pr:1} can be equivalently rewritten as
        \[
            (-1)^\ell\ddiv^\ell\left[ \left(\frac{\Gamma(n+1/q)}{\Gamma(n+1)}\right)^{p-1}\!-\frac{1}{(n+1)^{1/q}} \right]
              \ge 0
              \quad\text{ for all } n,\ell\in\N_0.
        \]
        By definition of complete monotonicity, see~\eqref{eq:def_CM}, and Lemma \ref{lem:MVT}, the last inequalities will be established once we show that the function
        \[
           f(x):=\left(\frac{\Gamma(x+1/q)}{\Gamma(x+1)}\right)^{p-1} \!-\frac{1}{(x+1)^{1/q}}
        \]
        is completely monotone on $\R_{+}$, which shall be done next.

        First, we prove that the function
        \[
            g(x):=(x+1)^{1/p}\,\frac{\Gamma(x+1/q)}{\Gamma(x+1)}
        \]
        is logarithmically completely monotone by a slight modification of the proof of Lemma~\ref{lem:div_g_LCM}. 
        Clearly, $g>0$ on $\R_{+}$ and 
        \[
            h(x):=-(\log{g(x)})'=\psi(x+1)-\psi(x+1/q)-\frac{1}{p(x+1)}
            \quad\text{ for all }x>0.
        \]
        Invoking integral representation \eqref{eq:psi_int_rep} and the elementary equality
        \[
            \frac{1}{x+1}=\int_0^\infty \ee^{-(x+1)t} \dd t
            \quad\text{ for }x>0,
        \]
        we may express the function $h$ as the Laplace transform 
        \[
            h(x)=\int_0^\infty \ee^{-xt} \rho(t) \dd t
            \quad\text{ for all }x>0,
            \,\text{ with }\,
            \rho(t):=\ee^{-t}\left(\frac{\ee^{t/p}-1}{1-\ee^{-t}}-\frac{1}{p}\right)
            \quad\text{ for all }t>0.
        \]
        Elementary computations show that the density $\rho>0$ on $\R_{+}$, from which we infer that $h$ is completely monotone, and so $g$ is logarithmically completely monotone. Consequently, $g^{p-1}$ is completely monotone. 

        From Gautschi's inequality~\eqref{eq:gautschi}, we additionally deduce that $g(x) > 1$ for all $x > 0$. Therefore, also the function $g^{p-1} - 1$ is completely monotone by definition. Since a product of completely monotone functions is a completely monotone function (see \cite[Cor.~1.6]{sch-son-von_12}),
        \[
        f(x)=(x+1)^{-1/q}\left(g^{p-1}(x)-1\right) \quad\text{ for all } x>0,
        \]
        and the function $x\mapsto(x+1)^{-1/q}$ is completely monotone on $\R_{+}$, we conclude that $f$ is completely monotone on $\R_{+}$, too.

        \emph{Step 2}: We further estimate the right-hand side of \eqref{eq:lem:LB_2_pr:1} as
        \begin{equation}
            \label{eq:lem:LB_2_pr:2}
            \sum_{j=0}^{\ell}(-1)^{j}\binom{\ell}{j} \frac{1}{(n-\ell+j+1)^{1/q}} \geq \left(\frac{1}{q}\right)_{\!\ell}\,\frac{1}{n^{\ell+1/q}}
        \end{equation}
        for all $n\ge\ell\ge2$ (note that the inequality fails for $\ell=1$). By applying the elementary formula 
        \[
            \frac{1}{x^\lambda}=\frac{1}{\Gamma(\lambda)}\int_0^\infty t^{\lambda-1} \ee^{-xt}\,\dd t,
        \]
        valid for all $\lambda,x>0$, and the binomial theorem, we find that 
        \begin{align*}
            \sum_{j=0}^{\ell}(-1)^{j}\binom{\ell}{j} \frac{1}{(n-\ell+j+1)^{1/q}} &= \frac{1}{\Gamma(1/q)}\,\sum_{j=0}^{\ell}(-1)^{j}\binom{\ell}{j} \int_0^\infty t^{1/q-1} \ee^{-(n-\ell+j+1)t} \,\dd t \\
            &= \frac{1}{\Gamma(1/q)}\,\int_0^\infty t^{1/q-1}\,(1-\ee^{-t})^\ell \ee^{(\ell-1)t}\,\ee^{-nt} \,\dd t.
        \end{align*}
        Next, we estimate the integrand as $(1-\ee^{-t})^\ell \ee^{(\ell-1)t} \ge t^\ell$, which follows from elementary inequalities
        \[
            \left(\frac{\ee^t-1}{t}\right)^{\!\ell} \geq \left(\frac{\ee^t-1}{t}\right)^{\!2} \geq \ee^t, \quad\text{ for all } t>0 \text{ and } \ell\geq2.
        \]
        The resulting estimate reads
        \[
        \sum_{j=0}^{\ell}(-1)^{j}\binom{\ell}{j} \frac{1}{(n-\ell+j+1)^{1/q}} \geq
        \frac{1}{\Gamma(1/q)}\int_0^\infty t^{\ell+1/q-1}\,\ee^{-nt} \,\dd t = \frac{\Gamma(\ell+1/q)}{\Gamma(1/q)}\frac{1}{n^{\ell+1/q}},
        \]
        which implies~\eqref{eq:lem:LB_2_pr:2}. Combining inequalities \eqref{eq:lem:LB_2_pr:1} and \eqref{eq:lem:LB_2_pr:2} concludes the proof.
    \end{proof}

\begin{proof}[Proof of inequality~\eqref{eq:ineq_improve_birman}:]
If $\ell\geq2$, inequality~\eqref{eq:ineq_improve_birman} readily follows from formula~\eqref{eq:optimal_rho_LB_form} and Lemmas~\ref{lem:LB_1} and~\ref{lem:LB_2}. Since Lemma~\ref{lem:LB_2} does not extend to $\ell=1$, this case must be treated separately.

If $\ell=1$, it is an easy exercise to expand~\eqref{eq:rho_1,p} into the convergent power series
\[
\rho^{(1,p)}_n = \left(\frac{1}{q}\right)^{\!p} \sum_{k=0}^\infty \frac{(p)_{k}}{p^k}\frac{1}{(k+1)!}\frac{1}{n^{k+p}}
\quad\text{ for all } n\in\N;
\]
cf. Remark~\ref{rem:conj_posit_coeff}. Notice that all coefficients of the power series are positive. Using only the very first one, we deduce inequality~\eqref{eq:ineq_improve_birman} for $\ell=1$.

The proof of inequality~\eqref{eq:ineq_improve_birman}, and hence Theorem~\ref{thm:optimal_rho_LB_form}, is complete.
\end{proof}

\section{Addendum: another choice for the parameter sequence}\label{sec:add}

In this final section, we consider the parameter sequence $\tilde{\g}^{(\ell,p)}$ defined by~\eqref{eq:def:g_tilde}. The corresponding weight $\tilde{\rho}^{(\ell,p)}$, given by~\eqref{eq:def:rho_tilde}, was shown to be an optimal discrete $p$-Hardy--Rellich--Birman weight for $\ell=1$ and $p>1$ in~\cite{fis-kel-pog_23} or $\ell\in\N$ and $p=2$ in~\cite{sta-wac_26acc}. Here, we complement these results by the following claim.

\begin{samepage}
\begin{theorem}\label{thm:tilde_g_opt}
Let one of the following conditions hold:
    \begin{enumerate}[(i)]
    \setlength{\itemsep}{1ex}
        \item $\ell\in\N$ and $p\ge2$ is an integer,
        \item $\ell\in\{3,4,5,6,7\}$ and $p>2$,
        \item $\ell\in\{1,2\}$ and $p>1$.
    \end{enumerate}
    Then $\tilde{\rho}^{(\ell,p)}$ is a discrete $p$-Hardy--Rellich--Birman weight, which is critical, non-attainable in $\H_0^\ell$, and optimal near infinity. 
\end{theorem}
\end{samepage}

We will show that, under the assumption of Theorem~\ref{thm:tilde_g_opt}, the sequence $\tilde{\g}^{(\ell,p)}$ satisfies assumptions \eqref{assump:1}--\eqref{assump:4}. To this end, we first examine properties of $\tilde{\g}^{(\ell,p)}$, as a function on $\R_{+}$, extended naturally by the formula
\begin{equation} \label{eq:g_real}
        \tilde{\g}^{(\ell,p)}(x):=x^{1-1/p}\prod_{j=1}^{\ell-1}(x-j), \quad \text{ for } x>0.
\end{equation}

By the combinatorial identity \cite[Eq.~(26.8.7)]{dlmf}, we have
\begin{equation} \label{eq:g:stirling}
    \tilde{\g}^{(\ell,p)}(x) = \sum_{j=1}^\ell s(\ell,j)\,x^{j-1/p},
\end{equation}
where 
\[
    s(\ell,j) := (-1)^{\ell+j} \sum_{1\leq i_{1}<\dots<i_{\ell-j}<\ell} i_{1}i_{2}\dots i_{\ell-j}
\]
are the (signed) Stirling numbers of the first kind; see \cite[\S26.8]{dlmf} for more details. Notice that $(-1)^{\ell+j}s(\ell,j)>0$ for all $1\leq j\leq\ell$. 

We similarly extend the definition of discrete divergence \eqref{eq:def_grad_div_Z} to functions of $x>0$ and put 
\begin{equation}
    \label{eq:func_tilde_g}
\tilde{g}(x):=\ddiv^\ell\,\tilde{\g}^{(\ell,p)}(x), \quad\text{ where }\;
            \ddiv^\ell\,\tilde{\g}^{(\ell,p)}(x):=\sum_{j=0}^\ell \binom{\ell}{j}(-1)^{\ell+j}\,\tilde{\g}^{(\ell,p)}(x+j),
\end{equation}
for all $x>0$.

\begin{lemma}\label{lem:tilde_g_CM}
    The function $\tilde{g}$ defined by~\eqref{eq:func_tilde_g} is completely monotone on $\R_+$.
\end{lemma}

\begin{proof}
 Let $k\in\N_0$ and $x>0$. We infer from Lemma~\ref{lem:MVT} that there exists $\xi\in(x,x+\ell)$ such that
 \[
        (-1)^k\tilde{g}^{(k)}(x)
          = (-1)^k\frac{\dd^k}{\dd x^k}\ddiv^\ell\,\tilde{\g}^{(\ell,p)}(x) 
          = (-1)^k\ddiv^\ell\,\frac{\dd^k \tilde{\g}^{(\ell,p)}}{\dd x^k}(x)
          = (-1)^k \frac{\dd^{\ell+k}\tilde{\g}^{(\ell,p)}}{\dd x^{\ell+k}}(\xi).
 \]
 Using~\eqref{eq:g:stirling}, we find that 
 \[
  (-1)^k\tilde{g}^{(k)}(x)=\sum_{j=1}^{\ell}c_{j,k}^{(\ell,p)}\,\xi^{j-\ell-k-1/p},
 \]
 where
 \[
        c_{j,k}^{(\ell,p)}:=(-1)^{\ell+j}s(\ell,j)\prod_{m=0}^{\ell+k-1}\left|j-\frac{1}{p}-m\right|>0
 \]
 for all $j\in\{1,\dots,\ell\}$. Consequently, the function $\tilde{g}$ is completely monotone.
\end{proof}

Next, we address the verification of \eqref{assump:1}--\eqref{assump:4} for $\tilde{\g}^{(\ell,p)}$. Without the logarithmic complete monotonicity of $\tilde{g}$, the approach successfully applied for $\g^{(\ell,p)}$, which used the complete monotonicity of $g^{p-1}$, is inapplicable here. Because Lemma~\ref{lem:tilde_g_CM} only establishes complete monotonicity, and since the complete monotonicity is not invariant under arbitrary positive powers, see \cite{van-hae_96, van-hae_97}, we cannot verify \eqref{assump:2}--\eqref{assumpp:3} without additional restrictions.

Our approach relies on the results of van Haeringen \cite{van-hae_96, van-hae_97}. For $N\in\N$, we adopt the notation:
\[
\mathcal{CM}^N := \{ f\in C^{N}(\R_+) \mid (-1)^k f^{(k)}(x)\geq 0 \;\text{ for all } k\in\{0,\dots,N\} \text{ and all } x>0\}.
\]
Further, for $\lambda>0$, we write $f\in\mathcal{CM}^N_{\lambda}$ if and only if $f^{\lambda}\in\mathcal{CM}^N$. The class of completely monotone functions $\mathcal{CM}$ coincides with the intersection of $\mathcal{CM}^N$ over all $N\in\N$. The following claims are proven in~\cite[Thm.~4]{van-hae_96} and \cite{van-hae_97}.

\begin{lemma}[van Haeringen] \label{lem:CM_order}
        The following inclusions hold true:
        \begin{enumerate}[1)]
            \item for $N\in\{0,1,\dots,5\}$ and $\lambda\geq\kappa>0$, we have $\mathcal{CM}^N_\kappa\subset\mathcal{CM}_\lambda^N$;
            \item for $N\in\{0,1,\dots,7\}$ and $\lambda>1$, we have $\mathcal{CM}\subset\mathcal{CM}_\lambda^N$.
        \end{enumerate}
\end{lemma}

\begin{proposition} \label{prop:old_g:assump} $\,$
        \begin{enumerate}[1)]
            \item For any $\ell\in\N$ and $p>1$, $\tilde{\g}^{(\ell,p)}\in H^{\ell}$ and satisfies \eqref{assump:1} and \eqref{assump:4}.
            \item If any condition (i)--(iii) from Theorem~\ref{thm:tilde_g_opt} holds, then $\tilde{\g}^{(\ell,p)}$ satisfies also \eqref{assump:2} and~\eqref{assumpp:3}.
        \end{enumerate}
\end{proposition}

\begin{proof}
First we prove claim 1). It follows immediately from~\eqref{eq:g_real} that $\tilde{\g}^{(\ell,p)}\in H^{\ell}$.

To verify~\eqref{assump:1}, it suffices to prove it for $k=\ell$. Indeed, by definition~\eqref{eq:def_grad_div_Z}, $\grad^\ell \tilde{\g}^{(\ell,p)}_n>0$ for all $n\ge\ell$ means that $\grad^{\ell-1}\tilde{\g}^{(\ell,p)}_n>\grad^{\ell-1}\tilde{\g}^{(\ell,p)}_{n-1}$ for all $n\geq\ell$. Since $\tilde{\g}^{(\ell,p)}\in H^{\ell}$, $\grad^{\ell-1}\tilde{\g}^{(\ell,p)}_{\ell-1}=0$, and so $\grad^{\ell-1}\tilde{\g}^{(\ell,p)}_{n}>0$ for all $n\geq\ell$. Iterating this argument, we show that
 $\grad^k\tilde{\g}^{(\ell,p)}_n>0$ for all $n\geq\ell$ and $k\in\{0,\dots,\ell\}$, which is \eqref{assump:1}.

It remains to check that $\grad^\ell \tilde{\g}^{(\ell,p)}_n>0$ for all $n\ge\ell$. To this end, we apply Lemma~\ref{lem:MVT} to~\eqref{eq:func_tilde_g}, which implies that for any $n\geq\ell$ there exists $\xi\in(n-\ell,n)$ such that
    \[
        \grad^{\ell}\tilde{\g}_{n}^{(\ell,p)}=\frac{\dd^{\ell}\tilde{\g}^{(\ell,p)}}{\dd x^{\ell}}(\xi)=\sum_{j=1}^{\ell}d_{j}^{(\ell,p)}\,\xi^{j-\ell-1/p},
    \]
where
    \[
        d_{j}^{(\ell,p)}:=(-1)^{\ell+j}s(\ell,j)\prod_{m=0}^{\ell-1}\left|j-\frac{1}{p}-m\right|>0.
    \]
Thus, \eqref{assump:1} holds. The verification of assumption~\eqref{assump:4} is immediate from~\eqref{eq:g:stirling} as $s(\ell,\ell)=1$.

Next we prove claim 2). The proof is carried out in three steps, each corresponding to the conditions (i), (ii), and (iii) from Theorem~\ref{thm:tilde_g_opt}, respectively.

\emph{Ad (i)} Let $\ell\in\N$ and $p\ge2$ be an integer. Recall that the set of completely monotone functions is closed under multiplication, see \cite[Cor.~1.6]{sch-son-von_12}. Consequently, from Lemma~\ref{lem:tilde_g_CM}, we infer that $\tilde{g}^{\,p-1}$ is completely monotone and, arguing the same as in \eqref{eq:p-lap_g_pos}, we show that 
\begin{equation}
\label{eq:strict_ineq_l=1}
(-1)^k\ddiv^k\bigl(\ddiv^\ell\tilde{\g}_n^{(\ell,p)}\bigr)^{p-1} > 0 \quad\text{ for all } k,n\in\N_0.
\end{equation}
In particular, $\tilde{\g}^{(\ell,p)}$ satisfies assumptions \eqref{assump:2} and \eqref{assumpp:3}.

\emph{Ad (ii)} Let $\ell\in\{3,4,5,6,7\}$ and $p>2$. By Lemma~\ref{lem:tilde_g_CM} and claim 2) of Lemma~\ref{lem:CM_order}, we have $\tilde{g}\in\mathcal{CM}\subset\mathcal{CM}_{p-1}^{\ell}$. Now, by applying Lemma~\ref{lem:MVT}, we deduce the inequalities
\[
(-1)^k\ddiv^k\bigl(\ddiv^\ell\tilde{\g}_n^{(\ell,p)}\bigr)^{p-1} \geq 0 \quad\text{ for all } n\in\N_0 \text{ and } k\in\{0,\dots,\ell\},
\]
which coincide with assumptions~\eqref{assump:2} and~\eqref{assumpp:3}.

\emph{Ad (iii)} Let $\ell=1$ or $\ell=2$ and $p>1$. Notice that, for $k=1$, the inequality \eqref{eq:strict_ineq_l=1} remains true for arbitrary $p>1$, which one verifies by Lemmas~\ref{lem:tilde_g_CM} and~\ref{lem:MVT}, bearing in mind that the inequalities in definition \eqref{eq:def_CM} are strict for non-constant completely monotone functions (see \cite[Rem.~1.5]{sch-son-von_12}). It means that
\begin{equation}
        \label{eq:stric_ineq_p_in_Z}
        \bigl(\ddiv^\ell\tilde{\g}^{(\ell,p)}_n\bigr)^{p-1}
          > \bigl(\ddiv^\ell\tilde{\g}^{(\ell,p)}_{n+1}\bigr)^{p-1}
          \quad\text{ for all }n\in\N_0.
\end{equation}
If $\ell=1$, we see from \eqref{eq:stric_ineq_p_in_Z} that \eqref{assumpp:3} holds even with strict inequality. Assumption \eqref{assump:2} is void for $\ell=1$.

For $\ell=2$, we shall make use of an auxiliary statement, proven in \cite[Thm.~7]{lor-new_83}, which states that, if $h\in\mathcal{CM}^{N+1}$, with $N\geq1$, then the function $h^{1/N}$ is convex. Put $h:=\ddiv^2\tilde{\g}^{(2,p)}$, then by Lemma~\ref{lem:tilde_g_CM} and the above property, we have
\[
        \left(h^{1/N}\right)'=\frac{h'}{Nh^{1-1/N}}\leq 0 \quad\text{ and }\quad \left(h^{1/N}\right)''\geq0 \quad\text{ on } \R_{+} \text{ for all } N\geq1.
\]
Therefore $h\in\mathcal{CM}_{1/N}^2$. Employing also claim 1) of Lemma~\ref{lem:CM_order}, we conclude that $h\in\mathcal{CM}_{\lambda}^2$ for any $\lambda>0$, hence, in particular, $h\in\mathcal{CM}_{p-1}^2$. As before, a quick application of Lemma~\ref{lem:MVT} shows that \eqref{assump:2} and \eqref{assumpp:3} hold for $\ell=2$. The proof of Proposition~\ref{prop:old_g:assump} is complete.
\end{proof}

\begin{proof}[Proof of Theorem~\ref{thm:tilde_g_opt}]
    It follows readily from Proposition~\ref{prop:old_g:assump} and Remark~\ref{rem:(A1-3)_impl_weak-opt}.
\end{proof}

\begin{remark}
Notice that the proof of Proposition~\ref{prop:old_g:assump} actually shows that assumption \eqref{assumpp:3} holds for $\tilde{\g}^{(\ell,p)}$ with the strict inequality for $\ell=1$ and $p>1$ or $\ell\in\N$ and integer $p\geq2$. By Theorem~\ref{thm:optimal_p-Hardy}, for these values of $\ell$ and $p$, the statement of Theorem~\ref{thm:tilde_g_opt} can be strengthened to the optimality of the discrete $p$-Hardy--Rellich--Birman weight $\tilde{\rho}^{(\ell,p)}$. In particular, for $\ell=1$, we extended the $p$-Hardy inequality, originally obtained in~\cite{fis-kel-pog_23}, to complex-valued sequences. In addition, we established that the weight $w^{\rm{FKP}}_p = \tilde{\rho}^{(1,p)}$, see~\eqref{eq:weight_fkp}, is not only null-critical, but in fact optimal in the sense of Definition~\ref{def:opt}.
\end{remark}

The proof technique of Proposition~\ref{prop:old_g:assump} does not extend to all $\ell\in\N$ and $p>1$; a specialized approach utilizing the precise structure of the parameter sequence $\tilde{\g}^{(\ell,p)}$ seems inevitable. Nevertheless, we conjecture
that all properties of $\tilde{\rho}^{(\ell,p)}$ proven in~\cite{fis-kel-pog_23, sta-wac_26, sta-wac_26acc} for the restricted cases $\ell=1$, $p>1$ or $\ell\in\N$, $p=2$ hold across the full parameter range $\ell\in\N$ and $p>1$. We formalize our expectation as the final conjecture.

\begin{conjecture}
    Let $\ell\in\N$ and $p>1$, and let $\tilde{\rho}^{(\ell,p)}$ be defined by~\eqref{eq:def:rho_tilde} and~\eqref{eq:def:g_tilde}.
    \begin{enumerate}[(i)]
        \item The sequence $\tilde{\rho}^{(\ell,p)}$ is an optimal discrete $p$-Hardy--Rellich--Birman weight.
        \item The sequence $\tilde{\rho}^{(\ell,p)}$ improves pointwise upon the classical $p$-Birman weight; cf.~\eqref{eq:ineq_improve_birman}.
        \item For all $n\geq\ell$, the terms $\tilde{\rho}_{n}^{(\ell,p)}$ admit a power series expansion in negative powers of $n$ with entirely non-negative coefficients; cf.~\eqref{eq:rho_ser_exp_pos}.
    \end{enumerate}
\end{conjecture}

\bibliographystyle{acm}

\end{document}